\documentclass[12pt,a4paper]{amsart}
\usepackage{graphicx}
\usepackage{tikz}
\usepackage{array}
\usepackage{amsmath}
\usepackage{amssymb}
\usepackage{latexsym}
\usepackage{epsfig}
\usepackage{blkarray}
\usepackage{bigstrut,mathdots}
\usepackage{mathtools}

\usepackage{tikz}
\usetikzlibrary{calc,shapes,intersections,patterns,angles,quotes,through,backgrounds,decorations.markings}

\usepackage{caption}

\def\lineWidth{0.4mm}
\usetikzlibrary{arrows}
\tikzset{
	whiteDot/.style={minimum size=20pt, draw,shape=circle,fill=white, text=black, inner sep=0.0pt}
}
\tikzset{
	blackDot/.style={minimum size=20pt, draw,shape=circle,fill=black, text=white, inner sep=0.0pt}
}

\tikzset{middlearrow/.style={
        decoration={markings,
            mark= at position 0.5 with {\arrow[scale=2]{#1}} ,
        },
        postaction={decorate}
    }
}

\usepackage[most]{tcolorbox}
\usepackage{tikz}
\usepackage{tkz-euclide}
\usetikzlibrary{calc,shapes,intersections,patterns,angles,quotes,through,backgrounds,decorations.markings}

\usetikzlibrary{shadows}
\tcbuselibrary{listings}
\tcbuselibrary{breakable}
\tcbuselibrary{skins}
\usetikzlibrary{shadings}
\usepackage{caption}
\usepackage{algorithm}
\usepackage{algpseudocode}
\makeatletter
\def\BState{\State\hskip-\ALG@thistlm}

\usepackage[noadjust]{cite}
\usepackage{amsmath,gauss}

\newcommand{\vol}{{\rm vol}}

\newcommand{\sgn}{{\rm sgn}}
\newcommand{\rank}{{\rm rank}}

\usetikzlibrary{calc,arrows,positioning}
\usetikzlibrary{matrix,decorations.pathreplacing,backgrounds,positioning}
\usetikzlibrary{calc}
\usepackage{amsmath,amsbsy,amsfonts,amsthm,latexsym,amsopn,amstext,amsxtra,euscript,amscd,mathrsfs}

\usepackage{mathdots}

\newtheorem{theorem}{Theorem}[section]

\newtheorem{corollary}[theorem]{Corollary}
\newtheorem{lemma}[theorem]{Lemma}

\newtheorem{question}[theorem]{Question}

\theoremstyle{definition}
\newtheorem{example}[theorem]{Example}
\newtheorem{remark}[theorem]{Remark}

\numberwithin{equation}{section}

\numberwithin{table}{section}

\numberwithin{figure}{section}
\def\Dj{D\kern-0.8em\char"16\kern-0.1em}

\begin{document}

\title[Signed graphs and inverses of their incidence matrices]{Signed graphs and inverses of their incidence matrices}

\author{Abdullah Alazemi}
\address{Department of Mathematics, Kuwait University, Safat 13060,
Kuwait} \email{abdullah.alazemi@ku.edu.kw}

\author{Milica An\dj{}eli\'c}
\address{Department of Mathematics, Kuwait University, Safat 13060,
Kuwait, and Department of Natural and Mathematical Sciences, State University of Novi Pazar, Serbia}
\email{milica.andelic@ku.edu.kw}

\author{Sudipta Mallik}
\address{Department of Mathematics and Statistics, Northern Arizona University, 801 S. Osborne Dr., PO Box: 5717, Flagstaff, AZ 86011, USA} \email{sudipta.mallik@nau.edu}

\subjclass[2000]{05C50, 15A09}

\date{\today}

\keywords{Moore-Penrose inverse; signed graphs; incidence matrix; Laplacian matrix}

\begin{abstract}
The Laplacian matrix $L$ of a signed graph $G$ may or may not be invertible. We present a combinatorial formula of the Moore-Penrose inverse of $L$. This is achieved by finding a combinatorial formula for the Moore-Penrose inverse of an incidence matrix of $G$. This work generalizes related known results on incidence and Laplacian matrices of an unsigned graph. Several examples are provided to show the usefulness of these combinatorial formulas.
\end{abstract}

\maketitle \medskip\noindent

\section{Introduction}
%\noindent Let $V(G)=\{1,2,\ldots, n\}$ and $E(G)=\{e_1,e_2, \ldots, e_m\}$, be the vertex, resp. edge sets of a graph $G=(V(G),E(G))$, and let $\sigma: E(G)\rightarrow \{-1,1\}$ be a mapping defined on $E(G)$. Then $\Gamma=(G,\sigma)$ is a {\it signed graph}. By $E^-(G)$ (resp. $E^+(G)$) we denote the set of negative (resp. positive) edges in $\Gamma$.

A signed graph is a graph in which each of its edges is positive or negative. More formally, a  {\it signed graph} $\Gamma$ is a pair $\Gamma=(G,\sigma)$ where $G=(V,E)$ is a graph, called the {\it underlying graph}, with vertex and edge sets $V=\{1,2,\ldots, n\}$ and $E=\{e_1,e_2, \ldots, e_m\}$ respectively and $\sigma: E\rightarrow \{-1,1\}$ is the {\it sign function} or {\it signature}. An edge $e_{\ell}$ is {\it positive} (resp. {\it negative}) if $\sigma(e_{\ell})=1$ (resp. $\sigma(e_{\ell})=-1$). By $E^-(\Gamma)$ (resp. $E^+(\Gamma)$) we denote the set of negative (resp. positive) edges in $\Gamma$.  Throughout this article, we consider only simple signed graphs, i.e., signed graphs that do not have loops and multiple edges.

A cycle $C$ in $\Gamma$ is {\it balanced} if it contains an even number of negative edges. $\Gamma$ is said to be {\it balanced} if all its cycles are balanced, otherwise it is {\it unbalanced}. The adjacency matrix of $\Gamma=(G,\sigma)$ is $A(\Gamma)=[a^{\sigma}_{ij}]  $ with $a^{\sigma}_{ij}=\sigma(ij)a_{ij}$, where $A=[a_{ij}]$ is the adjacency matrix of the underlying graph $G$. The corresponding Laplacian matrix $L(\Gamma)=D(G)-A(\Gamma)$, where $D(G)$ is the diagonal matrix of vertex degrees. Before we define incidence matrix of $\Gamma$ in details, note that $L(\Gamma)=NN^\intercal$ for any incidence matrix $N$ of $\Gamma$.

An incidence matrix $N=[n_{ij}]$ of $\Gamma$ is an $n\times m$ binary matrix  with rows indexed by vertices and columns indexed by edges such that each of its columns has exactly two nonzero entries as follows: 
\begin{enumerate}
    \item[(a)] when $e_{\ell}=\{i,j\}$ is a negative edge, $n_{i\ell}=n_{j\ell}$ and it is either $1$ or $-1$, and 
    \item[(b)] when $e_{\ell}=\{i,j\}$ is a positive edge, $n_{i\ell}=-n_{j\ell}$ and it is either $1$ or $-1$.
\end{enumerate}
Multiple incidence matrices can be associated with $\Gamma$. In particular, if $\Gamma=(G,\sigma)$ has all negative edges, then the incidence matrix of the underlying graph $G$ is an incidence matrix of $\Gamma$. Similarly, if $\Gamma=(G,\sigma)$ has all positive edges, then an oriented incidence matrix of the underlying graph $G$ is an incidence matrix of $\Gamma$. It is important to mention that an oriented incidence matrix $N=[n_{ij}]$ of $G$ gives each of its edges an orientation as follows \cite{BH,bap}: when $n_{i\ell}=-n_{j\ell}=1$, edge $e_{\ell}=\{i,j\}$ is oriented as $(i,j)$ which is directed from vertex $i$ to vertex $j$.

We get a fixed incidence matrix of $\Gamma=(G,\sigma)$ if each of its edges has a bi-direction in which each edge has two arrowheads at the ends. We define $n_{i\ell}=\eta(i,e_\ell)$ where $\eta: V\times E\to \{+1,-1, 0\}$ is defined as follows:
\begin{itemize}
\item[(i)] $\eta(i,e_\ell)=0$ if $i\notin e_\ell$,
\item[(ii)] $\eta(i,e_\ell)=-1$ (resp. $+1$) if an arrow at $i$ goes into (resp. out of) $i$ and $e_\ell=\{i,j\}$,
\item[(iii)] $\eta(i,e_\ell)\eta(j, e_\ell)=-\sigma(e_\ell)$ for $e_\ell=\{i,j\}$.
\end{itemize}
Note that the last two conditions imply that two arrows of a positive (resp. negative) edge have the same  direction (resp. opposite directions). In the literature of signed graphs the signs in  condition (ii) are reversed. We changed it to be consistent with the literature of oriented incidence matrix of an unsigned graph \cite{BH,bap}. This work is a generalization of work on graphs to signed graphs and we did not want to contradict the existing results on the Moore-Penrose inverse of an oriented matrix of a graph \cite{bap}.

For  $A\in \mathbb R^{n\times m}$,  a  matrix $X\in \mathbb R^{m\times n}$ is called the {\it Moore-Penrose inverse} of $A$ if 
\begin{equation}\label{1}
AXA=A,\,\,
XAX=X, \,\,(AX)^{\intercal}=AX,\,\, (XA)^{\intercal}=XA,
 \end{equation} and it is denoted by $A^{\dagger}.$

In 1965,  Ijiri first studied the Moore-Penrose inverse of the oriented incidence matrix of a graph in \cite{I}. Bapat did the same for the Laplacian and oriented incidence  matrices of trees \cite{bap}. Further research studied the same topic for different graphs such as distance regular graphs \cite{AB,ABE}. With the emergence of research on the signless Laplacian of graphs \cite{CRC, HM},  Hessert and Mallik studied the Moore-Penrose inverses of the incidence matrix and  signless Laplacian of a tree and an unicyclic graph in \cite{HM1, HM2}. Ipsen and Mallik studied wheel graphs in the same context \cite{Ipsen}. Alazemi, Alhalabi, and An\dj eli\'c did the same for a larger graph class namely bipartite graphs \cite{AOM}. Since a signed graph is a generalization of an unsigned graph, it is a natural open problem to extend the investigation regarding Moore-Penrose inverse to signed graphs. In this article we find combinatorial formula for the entries of the Moore-Penrose inverse of the Laplacian and incidence matrices of a signed graph. 

%The aim of this paper is to provide combinatorial formula for the entries of the Moore-Penrose inverse of the above mentioned matrices associated to signed graphs. Obtained results can be also seen as a generalization/extension of those of \cite{AOM, bap, HM1}, where the Moore-Penrose inverse of some particular cases of $N(\Gamma)$ have been covered.

The content of the present paper is organized as follows. 
In Section \ref{sec:2} we deliver several preliminary results. Section \ref{sec:3} covers combinatorial formulas for the Moore-Penrose inverse of incidence matrices of signed trees and unbalanced unicyclic signed graphs. Section \ref{sec:4} extends the results of Section \ref{sec:3} to incidence matrices of  connected signed graphs.  Results related to Laplacian matrix are subject of Section \ref{sec:5}. Several open problems are reported in Section \ref{sec:6}.

\section{Preliminaries}\label{sec:2}

\noindent We recall several  results from \cite{bap,BSKS}, essential for some initial considerations regarding the rank and minors of incidence matrices of signed graphs. 

\begin{lemma}{\rm (\cite[Lemma 3.4]{BSKS})}\label{rank}
Let $N$ be the incidence matrix of a connected signed graph $\Gamma$ of order $n$. Then 
$$\rank(N)=\begin{cases}
n-1& \text{ if \,$\Gamma$ is balanced,}\\
n& \text{ if \,$\Gamma$ is unbalanced.}
\end{cases}$$
\end{lemma}

For a matrix $N$, $N(i;\,)$ denotes the submatrix of $N$ with deleted row $i$.

\begin{lemma}{\rm (\cite[Lemma 3.7]{BSKS})}\label{det}
Let $N$ be the incidence matrix of a signed graph $\Gamma$. If $\Gamma$ is a tree, then $\det (N(i;\,))=\pm 1$, and if $\Gamma$ is an unbalanced unicyclic graph, then $\det (N(i;\,))=\pm 2.$
\end{lemma}

For an $n\times m$ real matrix $A$ and a positive integer $r\leq n,m$,   $\vol\, A$ denotes the square root of the sum of squares of all $r\times r$ minors of $A$ and 
$A(i_1,i_2,\ldots,i_r),\; 1\leq i_1<i_2<\cdots <i_r\leq m$ denotes the matrix obtained from $A$ by replacing all the columns of $A$ that are different from columns $i_1,i_2,\ldots,i_r$ by zero vectors. Then the Moore-Penrose inverse of $A$ can be computed in the following way.

\begin{theorem}{\rm (\cite{bap, AOM})}\label{MPDet}
Let $A$ be an $n\times m$ real matrix of rank $r>0$. Then
\begin{equation}\label{MPdet}
A^\dagger=\frac{1}{\vol^2 A}\sum_{(i_1,i_2,\ldots,i_r)\in \mathcal M} \vol^2 A(i_1,i_2,\ldots,i_r) A(i_1,i_2,\ldots,i_r)^\dagger, 
\end{equation}
where 
\[\mathcal M=\{(i_1,i_2,\ldots,i_r)\in \mathbb Z^r \;|\; 1\leq i_1<i_2<\cdots <i_r\leq m, \; \rank(A(i_1,i_2,\ldots,i_r))=r\}.\]
\end{theorem}

According to Lemma \ref{rank} the rank of an incidence matrix of a signed graph $\Gamma$ is either $n-1$, if $\Gamma$ is balanced or $n$ if $\Gamma$ is unbalanced. In addition,  by Lemma \ref{det}, maximal rank is attained through spanning trees, in case $\Gamma$ is balanced or through  unbalanced unicyclic graphs for $\Gamma$ balanced.

A signed $TU$-subgraph $H$ of $\Gamma$ is a subgraph whose components are trees or unbalanced unicyclic graphs with $c(H)$ denoting the number of unbalanced unicyclic components. By Lemmas \ref{rank} and \ref{det}, we have the following.

\begin{theorem}\label{vol}
Let $\Gamma$ be a connected signed graph on $n\geq 2$ vertices $1,2,\ldots,n$ with incidence matrix $N$. Let $H$ be a spanning signed subgraph of $\Gamma$ with incidence matrix $N_H$.
\begin{enumerate}
    \item[(a)] If $\Gamma$ is balanced and $H$ is a spanning tree of $\Gamma$, then $\rank(N)=\rank(N_H)=n-1$ and for all $i=1,2,\ldots,n$,
    \[\det(N_H(i;))=\pm 1,\]
   and consequently $\vol^2(N_H)=n$. If $H$ is any other kind of spanning subgraph of $\Gamma$ with $n-1$ edges, then $\det(N_H(i;))=0$
    for all $i=1,2,\ldots,n$.
    
    \item[(b)] If $\Gamma$ is unbalanced and $H$ is a spanning $TU$-subgraph of $\Gamma$ with $n$ edges consisting of $c(H)$ unbalanced unicyclic graphs,
    then $\rank(N)=\rank(N_H)=n$ and 
    \[\det(N_H)=\pm 2^{c(H)},\] 
    and consequently $\vol^2(N_H)=4^{c(H)}$.
    If $H$ is any other kind of spanning subgraph of $\Gamma$ with $n$ edges, then $\det(N_H)=0$.
\end{enumerate}
\end{theorem}

\begin{proof}
(a) Suppose $\Gamma$ is balanced and $H$ is a spanning tree of $\Gamma$. By Lemmas \ref{rank}, \ref{det}, $\rank(N)=\rank(N_H)=n-1$ and for $i \in \{1,2,\ldots,n\}$, $\det(N_H(i;))=\pm 1$. Then 
\[\vol^2(N_H)=\sum_{i=1}^n (\det(N_H(i;)))^2=n.\]
Now suppose $H$ is a spanning subgraph of $\Gamma$ with $n-1$ edges that is not a spanning tree of $\Gamma$. Then $H$ has more than one connected (balanced) component and consequently $\rank(N_H(i;))<n-1$. Thus  $\det(N_H(i;))=0$.\\ 

(b) Suppose $\Gamma$ is unbalanced and $H$ is a spanning $TU$-subgraph of $\Gamma$ with $n$ edges consisting of $c(H)$ unbalanced unicyclic graphs. By Lemma \ref{rank}, $\rank(N)=\rank(N_H)=n$. By Lemma \ref{det} the determinant of the submatrix of $N_H$ corresponding to each unbalanced unicyclic graph is $\pm 2$. Thus $\det(N_H)=(\pm 2)^{c(H)}$. Then 
\[\vol^2(N_H)= (\det(N_H))^2=4^{c(H)}.\]
Now suppose $H$ is a spanning subgraph of $\Gamma$ with $n$ edges whose connected components are not all unbalanced unicyclic graphs. If $H$ has a component $H_1$ such that $|E(H_1)|<|V(H_1)|$, then the columns of $N_H$ corresponding to $V(H_1)$ are linearly dependent and consequently $\det(N_H)=0$. Now suppose $H$ has $t$ unicyclic components $H_1,H_2,\ldots,H_t$. Since not all of $H_1,H_2,\ldots,H_t$ are unbalanced, say $H_1$ is balanced. Then the determinant of the submatrix of $N_H$ corresponding to $H_1$ is $0$. Thus 
\[\det(N_H)=\pm \prod_{i=1}^n \det(N_{H_i})=0.\]
\end{proof}

In what follows by $\tau(\Gamma)$  we denote the number of spanning trees of $\Gamma$, while by $\mathcal{TU}_n(\Gamma)$ we denote the set of $TU$-subgraphs of $\Gamma$ on $n$ edges.

\begin{corollary}
Let $\Gamma$ be a connected signed graph on $n$ vertices with an incidence matrix $N$ and Laplacian matrix $L$. 
\begin{enumerate}
\item[(a)] If $\Gamma$ is balanced, then $\det(L)=0$ and $\vol^2(N)=n\tau(\Gamma)$.

\item[(b)] If $\Gamma$ is unbalanced, then 
\[\det(L)=\vol^2(N)=\sum_{H\in \mathcal{TU}_n(\Gamma)} 4^{c(H)}.\]
\end{enumerate}
Moreover, $\frac{\det(L)}{4}\geq $ the number of unbalanced unicyclic graphs in $G$ and the equality holds if and only if $G$ is balanced or an unbalanced unicyclic graph.
\end{corollary}

% \begin{proof}
% Needed?
% \end{proof}

In order to implement \eqref{MPdet} to compute $N^{\dagger}$ as main ingredients we need formulas for the Moore-Penrose inverse of the incidence matrix of a tree as well of an unbalanced unicyclic graph.
These are covered in subsequent sections.

\section{Signed trees and unbalanced unicyclic signed graphs}\label{sec:3}

\noindent We first introduce necessary notation. For a given signed tree $T$ on $n$ vertices $1,2,\ldots,n$, we introduce the {\it head} and {\it tail components} in $T\setminus e$ for any edge $e$. 

\begin{itemize}
\item If $e_i=\{l_i,m_i\}$  such that $l_i < m_i$ is  a negative edge, then the \textit{head component} of $T\setminus e_i$, denoted by $T_h(e_i)$, is the connected component of $T\setminus e_i$ containing the vertex $m_i$. The \textit{tail component} of $T\setminus e_i$, denoted by $T_t(e_i)$, is the connected component of $T\setminus e_i$ containing the vertex $l_i$. 

\item If $e_i=(l_i,m_i)$ is a positive edge, directed from $l_i$ to $m_i$, then the \textit{head component} of $T\setminus e_i$ is the component containing $m_i$, while the \textit{tail component} contains  $l_i$.
\end{itemize}

We recall that $|T_h(e_i)|+|T_t(e_i)|=n$, while according to  \cite[Lemma 2.4]{HM1}
$$\sum_{e_p\in E_h(i)} |T_t(e_p)|+\sum_{e_q\in E_T(i)} |T_h(e_q)|=n-1,$$
where $$E_h(i)=\{e_k\in E(G): e_k \text{ is incident at } i\text{ and } i\in T_h(e_k)\},$$  $$E_t(i)=\{e_k\in E(G): e_k \text{ is incident at } i\text{ and } i\in T_t(e_k)\}.$$

The  shortest path between a vertex $j$ and an edge $e_i$, that does not contain $e_i$ is denoted by $P_{e_i-j}$. The product of $\sigma(e_k)$, for $e_k$ being on  the path $P_{e_i-j}$ is denoted by $\sgn (P_{e_i-j})$. 

For a connected balanced graph $\Gamma$ on $n$ vertices $1,2,\ldots,n$, we define a new $n\times n$ matrix $S=[s_{ij}]$, we call the {\it path sign matrix} of $\Gamma$, $s_{ii}=1$ for all $i$ and $s_{ij}=1$ ($-1$) if there is a positive (resp. negative) path between distinct vertices $i$ and $j$. The corresponding sign we denote by $\sgn(P_{i-j})$. The role of $S$ in computation of Moore-Penrose inverse of balanced graphs is justified in following theorem.

\begin{theorem}\label{ps}
Let $\Gamma=(G,\sigma)$ be a connected signed graph on $n\geq 2$ vertices $1,2,\ldots,n$ with an incidence matrix $N$. 
\begin{enumerate}
    \item[(a)] If $\Gamma$ is unbalanced, then $NN^\dagger=I_n$. 
    
    \item[(b)] If $\Gamma$ is balanced, then 
    \[NN^\dagger=I_n-\frac{1}{n}S,\]
    where $S$ is the path sign matrix of $\Gamma$.
\end{enumerate}
\end{theorem}
\begin{proof}
Since $NN^\dagger N=N$, $(I_n-NN^\dagger)N=O_{n,m}$. Then each row of $I_n-NN^\dagger$ is orthogonal to each column of $N$. Suppose $x^T=[x_1,x_2,\ldots,x_n]$ is orthogonal to each column of $N$. Then $x_i-x_j=0$ for any positive edge $\{i,j\}$ and $x_i+x_j=0$ for any negative edge $\{i,j\}$  in $\Gamma$. Therefore, if there is a path $P_{i-j}$ between $i$ and $j$, then $x_i=\sgn(P_{i-j}) x_j$.\\

(a) If $\Gamma$ is unbalanced, then by Lemma \ref{rank}, $\rank\,N=n$ and $N$ is invertible, hence $N^\dagger=N^{-1}$.
%(a) Suppose $G$ is unbalanced with a negative cycle $C$. Let $v$ and $k$ be two distinct vertices on $C$. Then there is a positive and a negative path between $k$ and $v$. Therefore, $x_v=x_k$ and $x_v=-x_k$ which implies $x_v=0$. Thus $x_v=0$ for any vertex $v$ on $C$. Since $G$ is connected, there is a path between a vertex $t$ not on $C$ and a vertex $v$ on $C$. Thus $x_t=\pm x_v=0$ for any vertex $t$ not on $C$. Thus  $x^T=0^\intercal$ which implies $I_n-NN^\dagger=O_n$, i.e., $NN^\dagger=I_n$.\\

(b) Suppose $\Gamma$ is balanced. Then $\rank\,N=n-1$ and all paths between two distinct vertices in $G$ have the same sign. For a fixed vertex $i$ of $G$, $x_v=\sgn(P_{i-v}) x_i$ for any vertex $v$ of $G$. Thus 
\[[x_1,x_2,\ldots,x_n]=x_i[\sgn(P_{i-1}),\sgn(P_{i-2}),\ldots,\sgn(P_{i-n})].\]
Consequently any row of $I_n-NN^\dagger$ is a multiple of $[\sgn(P_{i-1}),\sgn(P_{i-2}),\ldots,\sgn(P_{i-n})]$. 
 Suppose for each $i=1,2,\ldots,n$, $i$th row of  $I_n-NN^\dagger$ is 
$$c_i[\sgn(P_{i-1}),\sgn(P_{i-2}),\ldots,\sgn(P_{i-n})].$$
Note that if $c_1=c_2=\cdots=c_n=0$, then $NN^\dagger=I_n$ and consequently $n=\rank(NN^\dagger)\leq \rank(N)=n-1$, a contradiction. Thus $c_i\neq 0$ for some $i$ and by symmetry of $NN^\dagger$, $c_i\neq 0$ for all $i=1,2,\ldots,n$. Now observe that  $(I_n-NN^\dagger)^2=I_n-NN^\dagger$ because $N^\dagger NN^\dagger=N^\dagger$. Then the $(i,i)$-entry of $(I_n-NN^\dagger)^2$, that is $c_i^2n$, is same as the $(i,i)$-entry of $I_n-NN^\dagger$, which is $c_i \sgn(P_{i-i})=c_i$. Thus $c_i=\frac{1}{n}$ and the $(i,j)$-entry of $I_n-NN^\dagger$ is $\frac{\sgn(P_{i-j})}{n}$.   
\end{proof}

We next give combinatorial expression for the Moore-Penrose inverse of the incidence matrix of a tree.

\begin{theorem}\label{tree}
Let $\Gamma=(T,\sigma)$ be a signed tree on $n\geq 2$ vertices $1,2,\ldots,n$ and $n-1$ edges $e_1,e_2,\ldots,e_{n-1}$. Let $N$ be an incidence matrix of  $\Gamma$. Then the Moore-Penrose inverse $N^\dagger=[n_{i,j}^\dagger]$ of $N$ is given by
\begin{equation}\label{H}
n_{i,j}^\dagger=
 \frac{\sgn(P_{e_i-j})w(e_i)}{n}\begin{cases}
 |T_h(e_i)| & \text{ if } j \in T_t(e_i)\\
 -\sigma(e_i)|T_t(e_i)| & \text{ if } j \in T_h(e_i),
 \end{cases} 
\end{equation}
where $$w(e_i)= \begin{cases}
-1 & \text{if $e_i$ is negative with both negative ends}\\
\,\,\,\,1& \text{otherwise.}
\end{cases}$$

\end{theorem}

\begin{proof}
We first consider the case, when every negative edge has both positive ends.
Let $H=[h_{i,j}]$ be the  $(n-1)\times n$ matrix given by \eqref{H}. Then,

%\[
%h_{i,j}=
% \frac{\sgn(P_{e_i-j})}{n}\begin{cases}
% |G_h(e_i)| & \text{ if } j \in G_t(e_i)\\
% -\sgn(e_i)|G_t(e_i)| & \text{ if } j \in G_h(e_i), 
% \end{cases} 
%\]
%i.e.,
\begin{equation}
h_{i,j}=
 \frac{\sgn(P_{e_i-j})}{n}\begin{cases}
 |T_h(e_i)| & \text{ if } j \in T_t(e_i)\\
 |T_t(e_i)| & \text{ if } j \in T_h(e_i),\; e_i\in E^-(\Gamma)\\
 -|T_t(e_i)| & \text{ if } j \in T_h(e_i),\; e_i\in E^+(\Gamma). 
 \end{cases} 
\end{equation}

We show that $H=N^\dagger$, by  proving $HN=I_{n-1}$.

For $i,j\in \{1,\ldots,n-1\}$, suppose $e_i=\{l,m\}$ where $l < m$ and $e_j=\{r,s\}$. For $N=[n_{i,j}]$, the $(i,j)$-entry of $HN$ is given by
\[(HN)_{i,j}=\sum^{n}_{k=1} h_{i,k}n_{k,j} = 
\begin{cases} 
h_{i,r}+h_{i,s} & \text{ if } e_j \in E^-(\Gamma)\\
h_{i,r}-h_{i,s} & \text{ if } e_j=(r,s) \in E^+(\Gamma)\\
-h_{i,r}+h_{i,s} & \text{ if } e_j=(s,r) \in E^+(\Gamma).
\end{cases} 
\]
The values of $(HN)_{i,j}$ are computed through the following cases and subcases.

\noindent{\it Case 1.} $i=j$.\\
First note that $\sgn(P_{e_i-l})=\sgn(P_{e_i-m})=1$ (being the empty product).\\
Let $e_i \in E^-(\Gamma)$. Since $l<m$, $l\in T_t(e_i)$  and $m \in T_h(e_i)$. Then 
\[(HN)_{i,i}=\frac{\sgn(P_{e_i-l})}{n}|T_h(e_i)|+\frac{\sgn(P_{e_i-m})}{n}|T_t(e_i)|=\frac{1}{n} |T_h(e_i)| +\frac{1}{n} |T_t(e_i)|=\frac{1}{n}n=1.\]

Let $e_i \in E^+(\Gamma)$. Since $e_j=(l,m)$, then
 $l\in T_t(e_i)$, $m \in T_h(e_i)$, and
\[(HN)_{i,i}=\frac{\sgn(P_{e_i-l})}{n}|T_h(e_i)|-\frac{-\sgn(P_{e_i-m})}{n}|T_t(e_i)|=\frac{1}{n} |T_h(e_i)| +\frac{1}{n} |T_t(e_i)|=\frac{1}{n}n=1.\]

If $e_i=(m,l)$, then
 $m\in T_t(e_i)$, $l \in T_h(e_i)$, and 
\[(HN)_{i,i}=-\frac{-\sgn(P_{e_i-l})}{n}|T_t(e_i)|+\frac{\sgn(P_{e_i-m})}{n}|T_h(e_i)|=\frac{1}{n} |T_t(e_i)| +\frac{1}{n} |T_h(e_i)|=\frac{1}{n}n=1.\]

\noindent{\it Case 2.} $i\neq j$.\\
Since $e_i \not= e_j=\{r,s\}$ and $T$ is a tree, either both $r,s\in T_h(e_i)$ or both $r,s\in T_t(e_i)$. \\

%Without loss of generality 
\noindent{\it Subcase (a)} $r,s\in T_t(e_i)$.\\
Either $\sgn(P_{e_i-r})=\sgn(P_{l-r})=\sgn(P_{e_i-s})\sigma(e_j)$ or $\sgn(P_{e_i-s})=\sgn(P_{l-s})=\sgn(P_{e_i-r})\sigma(e_j)$. \\

If $e_j \in E^-(\Gamma)$, then $\sgn(P_{e_i-r})$ and $\sgn(P_{e_i-s})$ have the opposite signs and
\[(HN)_{i,j}=\frac{\sgn(P_{e_i-r})}{n}|T_h(e_i)|+\frac{\sgn(P_{e_i-s})}{n}|T_h(e_i)|=0.\]

If $e_j \in E^+(\Gamma)$, then $\sgn(P_{e_i-r})$ and $\sgn(P_{e_i-s})$ have the same sign and
\[(HN)_{i,j}= \frac{\sgn(P_{e_i-r})}{n}|T_h(e_i)|-\frac{\sgn(P_{e_i-s})}{n}|T_h(e_i)|=0.\]

\noindent{\it Subcase (b)} $r,s\in T_h(e_i)$.\\
Either $\sgn(P_{e_i-r})=\sgn(P_{m-r})=\sgn(P_{e_i-s})\sigma(e_j)$ or $\sgn(P_{e_i-s})=\sgn(P_{m-s})=\sgn(P_{e_i-r})\sigma(e_j)$. \\

If $e_j \in E^-(\Gamma)$, then $\sgn(P_{e_i-r})$ and $\sgn(P_{e_i-s})$ have the opposite signs and
\[(HN)_{i,j}=h_{i,r}+h_{i,s}= \frac{-\sgn(P_{e_i-r})}{n}|T_t(e_i)|+\frac{-\sgn(P_{e_i-s})}{n}|T_t(e_i)|=0.\]

If $e_j \in E^+(\Gamma)$, then $\sgn(P_{e_i-r})$ and $\sgn(P_{e_i-s})$ have the same sign and
\[(HN)_{i,j}=h_{i,r}-h_{i,s}= \frac{\sgn(P_{e_i-r})}{n}|T_t(e_i)|-\frac{\sgn(P_{e_i-s})}{n}|T_t(e_i)|=0.\]

Thus $HN=I_{n-1}$, $HN$ is symmetric, $HNH=H$, and $NHN=N$. It remains to show that $NH$ is symmetric as well. We show it by proving that
%\[NH=I_n-\frac{1}{n}S,\] 
%i.e. $(NH)_{ij}=-\frac{\sgn(P_{i-j})}{n}$. 
\[(NH)_{ij}=-\frac{\sgn(P_{i-j})}{n}.\]
We distinguish the following subsets of $E^+(\Gamma)$, resp. $E^-(\Gamma)$.
Let \begin{align*}
E^+_h(i)&:=\{e_k\in E^+(\Gamma)|\,e_k \,\mbox{is incident at}\,i \,\mbox{and}\, i\in T_h(e_k)\},\\
E^+_t(i)&:=\{e_k\in E^+(\Gamma)|\,e_k \,\mbox{is incident at}\,i \,\mbox{and}\, i\in T_t(e_k)\},\\
E^-_h(i)&:=\{e_k\in E^-(\Gamma)|\,e_k \,\mbox{is incident at}\,i \,\mbox{and}\, i\in T_h(e_k)\},\\
E^+_t(i)&:=\{e_k\in E^-(\Gamma)|\,e_k \,\mbox{is incident at}\,i \,\mbox{and}\, i\in T_t(e_k)\}.
\end{align*}
The values of $$(NH)_{ij}=-\sum_{k:k\in E^+_h(i)} h_{kj}+\sum_{k:k\in E^+_t(i)} h_{kj}+\sum_{k:k\in E^-_h(i)} h_{kj}+\sum_{k:k\in E^-_t(i)} h_{kj}$$
are computed in subsequent cases.

\noindent{\it Case 1.} Let  $i\neq j$, and suppose that $i-j$ path in $T$ uses the edge $e_l\in E^+_h(i).$ Then 
$$nh_{kj}=\begin{cases}
\sgn(P_{e_k-j}) |T_h(e_k)|&\mbox{if}\, \, k=l;\\
-\sgn(P_{e_k-j})|T_t(e_k)|&\mbox{if}\, \, e_k\in E^+_h(i), k\neq l;\\
 \sgn(P_{e_k-j})|T_h(e_k)|&\mbox{if}\, \, e_k\in E^+_t(i), k\neq l;\\
\sgn(P_{e_k-j})|T_t(e_k)|&\mbox{if}\, \, e_k\in E^-_h(i), k\neq l;\\
\sgn(P_{e_k-j})|T_h(e_k)|&\mbox{if}\, \, e_k\in E^-_t(i), k\neq l,
\end{cases}$$
and \begin{align*}
n(NH)_{ij}&=-\sgn(P_{e_l-j}) |T_h(e_l)|+\sum_{k:k\in E^+_h(i),\atop k\neq l}\sgn(P_{e_k-j}) |T_t(e_k)| +\sum_{k:k\in E^+_t(i),\atop k\neq l}  \sgn(P_{e_k-j}) |T_h(e_k)| \\&+\sum_{k:k\in E^-_h(i),\atop k\neq l} \sgn(P_{e_k-j})|T_t(e_k)|+\sum_{k:k\in E^-_t(i),\atop k\neq l} \sgn(P_{e_k-j})|T_h(e_k)|\\&=\sgn(P_{i-j})\Big(-|T_h(e_l)|+\sum_{k:k\in E^+_h(i),\atop k\neq l}|T_t(e_k)| +\sum_{k:k\in E^+_t(i),\atop k\neq l}  |T_h(e_k)| +\sum_{k:k\in E^-_h(i),\atop k\neq l} |T_t(e_k)|\\ &+\sum_{k:k\in E^-_t(i),\atop k\neq l} |T_h(e_k)|\Big)\\&=-\sgn(P_{i-j})(-|T_h(e_l)|-|T_t(e_l)|+\sum_{e\in E_h(i)}|T_t(e)|+\sum_{e\in E_t(i)}|T_h(e)|\\&=-n+(n-1)=-1=-\sgn(P_{i-j}).
\end{align*}

\noindent{\it Case 2.} Let  $i\neq j$, and suppose that $i-j$ path in $T$ uses the edge $e_r\in E^+_t(i).$ Then 
$$nh_{kj}=\begin{cases}
 -\sgn(P_{e_k-j})|T_t(e_k)|&\mbox{if}\, \, k=r;\\
-\sgn(P_{e_k-j})|T_t(e_k)|&\mbox{if}\, \, e_k\in E^+_h(i), k\neq r;\\
 \sgn(P_{e_k-j})|T_h(e_k)|&\mbox{if}\, \, e_k\in E^+_t(i), k\neq r;\\
\sgn(P_{e_k-j})|T_t(e_k)|&\mbox{if}\, \, e_k\in E^-_h(i), k\neq r;\\
\sgn(P_{e_k-j})|T_h(e_k)|&\mbox{if}\, \, e_k\in E^-_t(i), k\neq r
\end{cases}$$
and \begin{align*}
n(NH)_{ij}&=-\sgn(P_{e_r-j})|T_t(e_r)|+\sum_{k:k\in E^+_h(i),\atop k\neq r}\sgn(P_{e_k-j})|T_t(e_k)| +\sum_{k:k\in E^+_t(i),\atop k\neq r} \sgn(P_{e_k-j}) |T_h(e_k)|\\& +\sum_{k:k\in E^-_h(i),\atop k\neq r} \sgn(P_{e_k-j})|T_t(e_k)|\sum_{k:k\in E^-_t(i),\atop k\neq r} \sgn(P_{e_k-j})|T_h(e_k)|\\
&=\sgn(P_{i-j})\Big(-|T_t(e_r)|+\sum_{k:k\in E^+_h(i),\atop k\neq r}|T_t(e_k)| +\sum_{k:k\in E^+_t(i),\atop k\neq r}  |T_h(e_k)| +\sum_{k:k\in E^-_h(i),\atop k\neq r} |T_t(e_k)|\\ &+\sum_{k:k\in E^-_t(i),\atop k\neq r} |T_h(e_k)|\Big)=\sgn(P_{i-j})(-n+n-1)=-\sgn(P_{i-j}).
\end{align*}

\noindent{\it Case 3.} Let  $i\neq j$, and suppose that $i-j$ path in $T$ uses the edge $e_p\in E^-_h(i).$ Then 
$$nh_{kj}=\begin{cases}
 \sgn(P_{e_p-j})|T_h(e_p)|&\mbox{if}\, \, k=p;\\
\sgn(P_{e_k-j})|T_t(e_k)|&\mbox{if}\, \, e_k\in E^-_h(i), k\neq p;\\
 \sgn(P_{e_k-j})|T_h(e_k)|&\mbox{if}\, \, e_k\in E^-_t(i), k\neq p;\\
-\sgn(P_{e_k-j})|T_t(e_k)|&\mbox{if}\, \, e_k\in E^+_h(i), k\neq p;\\
\sgn(P_{e_k-j})|T_h(e_k)|&\mbox{if}\, \, e_k\in E^+_t(i), k\neq p
\end{cases}$$
and \begin{align*}
n(NH)_{ij}&=\sgn(P_{e_p-j})|T_h(e_p)|+\sum_{k:k\in E^+_h(i),\atop k\neq p}\sgn(P_{e_k-j})|T_t(e_k)| +\sum_{k:k\in E^+_t(i),\atop k\neq p} \sgn(P_{e_k-j}) |T_h(e_k)| \\ &+\sum_{k:k\in E^-_h(i),\atop k\neq p} \sgn(P_{e_k-j})|T_t(e_k)|+\sum_{k:k\in E^-_t(i),\atop k\neq p} \sgn(P_{e_k-j})|T_h(e_k)|\\&=
\sgn(P_{i-j})\Big(-|T_h(e_p)|+\sum_{k:k\in E^+_h(i),\atop k\neq p}|T_t(e_k)| +\sum_{k:k\in E^+_t(i),\atop k\neq p}  |T_h(e_k)| \\ &+\sum_{k:k\in E^-_h(i),\atop k\neq p}|T_t(e_k)|+\sum_{k:k\in E^-_t(i),\atop k\neq p} |T_h(e_k)|\Big)=-\sgn(P_{i-j}).
\end{align*}

\noindent{\it Case 4.} Let  $i\neq j$, and suppose that $i-j$ path in $T$ uses the edge $e_q\in E^-_t(i).$ Then 
$$nh_{kj}=\begin{cases}
 \sgn(P_{e_q-j})|T_t(e_q)|&\mbox{if}\, \, k=q;\\
\sgn(P_{e_k-j})|T_t(e_k)|&\mbox{if}\, \, e_k\in E^-_h(i), k\neq q;\\
 \sgn(P_{e_k-j})|T_h(e_k)|&\mbox{if}\, \, e_k\in E^-_t(i), k\neq q;\\
- \sgn(P_{e_q-j})|T_t(e_k)|&\mbox{if}\, \, e_k\in E^+_h(i), k\neq q;\\
 \sgn(P_{e_q-j})|T_h(e_k)|&\mbox{if}\, \, e_k\in E^+_t(i), k\neq q
\end{cases}$$
and \begin{align*}
n(NH)_{ij}&=\sgn(P_{e_q-j})|T_t(e_q)|+\sum_{k:k\in E^+_h(i),\atop k\neq q} \sgn(P_{e_q-j})|T_t(e_k)| +\sum_{k:k\in E^+_t(i),\atop k\neq q}  \sgn(P_{e_q-j}) |T_h(e_k)| \\ &+\sum_{k:k\in E^-_h(i),\atop k\neq q} \sgn(P_{e_k-j})|T_t(e_k)|+\sum_{k:k\in E^-_t(i),\atop k\neq q} \sgn(P_{e_k-j})|T_h(e_k)|\\
&=\sgn(P_{i-j})\Big(-|T_t(e_q)|+\sum_{k:k\in E^+_h(i),\atop k\neq q}|T_t(e_k)| +\sum_{k:k\in E^+_t(i),\atop k\neq q}  |T_h(e_k)| \\ &+\sum_{k:k\in E^-_h(i),\atop k\neq q} |T_t(e_k)|+\sum_{k:k\in E^-_t(i),\atop k\neq q}|T_h(e_k)|\Big)=-\sgn(P_{i-j}).
\end{align*}
 
If $i=j$, then
\begin{align*}(NH)_{ii}&=-\sum_{k:k\in E^+_h(i)} h_{ki}+\sum_{k:k\in E^+_t(i)} h_{ki}+\sum_{k:k\in E^-_h(i)} h_{ki}+\sum_{k:k\in E^-_t(i)} h_{ki}\\&=\sum_{k:k\in E^+_h(i)}|T_t(e_k)| +\sum_{k:k\in E^+_t(i)} |T_h(e_i)|+\sum_{k:k\in E^-_h(i)} |T_t(e_i)|+\sum_{k:k\in E^-_t(i)} |T_h(e_i)|
=n-1
\end{align*}

If a negative edge has both negative ends, then the corresponding row in $N^\dagger$ is multiplied by $-1$.
\end{proof}

\begin{example}
Consider the signed tree depicted in Figure \ref{fig:M tree}.
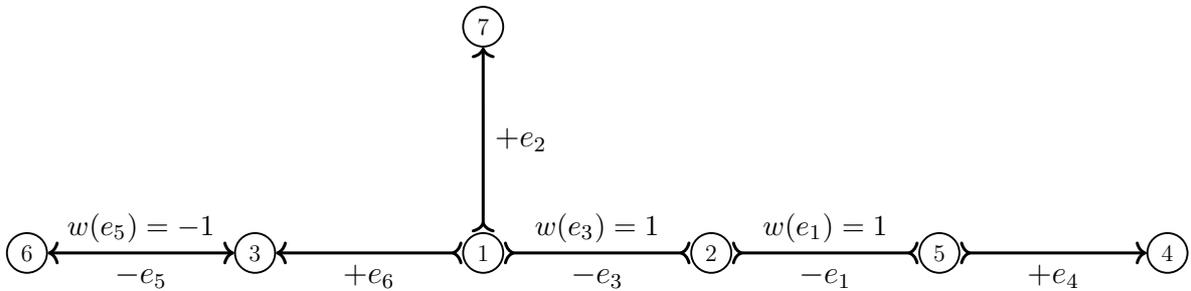
\begin{figure}[h]
	\centering
	    \def\NodeSC{0.75}
	    \def\Hdis{4.0}
	    \def\Vdis{4.0}
	    \begin{tikzpicture}[line width=0.25mm, baseline, scale=0.75]
	      \node[whiteDot,scale=\NodeSC] (A6) at (0,0) {\bf $6$};
		  	\path (A6) +(0:\Hdis) node[whiteDot,scale=\NodeSC] (A3) {\bf $3$};
		  	\path (A3) +(0:\Hdis) node[whiteDot,scale=\NodeSC] (A1) {\bf $1$};
	        \path (A1) +(0:\Hdis) node[whiteDot,scale=\NodeSC] (A2) {\bf $2$};
	        \path (A1) +(90:\Vdis) node[whiteDot,scale=\NodeSC] (A7) {\bf $7$};
	        \path (A2) +(0:\Hdis) node[whiteDot,scale=\NodeSC] (A5) {\bf $5$};
	        \path (A5) +(0:\Hdis) node[whiteDot,scale=\NodeSC] (A4) {\bf $4$};
	      \begin{pgfonlayer}{background}
	        \draw[line width=\lineWidth,<->] (A6) to node[midway, below] {\emph{$-e_5$}} (A3);
			\node[above] at ($(A6)!.5!(A3)$) {\small{$w(e_5) = -1$}};
	        \draw[line width=\lineWidth,<-<] (A3) to node[midway, below] {\emph{$+e_6$}} (A1);
	        \draw[line width=\lineWidth,>->] (A1) to node[midway, right] {\emph{$+e_2$}} (A7);
	        \draw[line width=\lineWidth,>-<] (A1) to node[midway, below] {\emph{$-e_3$}} (A2);
			\node[above] at ($(A1)!.5!(A2)$) {\small{$w(e_3) = 1$}};
	        \draw[line width=\lineWidth,>-<] (A2) to node[midway, below] {\emph{$-e_1$}} (A5);
			\node[above] at ($(A2)!.5!(A5)$) {\small{$w(e_1) = 1$}};
	        \draw[line width=\lineWidth,>->] (A5) to node[midway, below] {\emph{$+e_4$}} (A4);
	      \end{pgfonlayer}
	    \end{tikzpicture}
		\vspace{1cm}
		\caption{The smallest asymmetric tree.}
	\end{figure}\label{fig:M tree}

Then the incidence matrix 
$$N= \left[\begin{array}{rrrrrr}
0 & 1 & 1 & 0 & 0 & 1 \\
1 & 0 & 1 & 0 & 0 & 0 \\
0 & 0 & 0 & 0 & -1 & -1 \\
0 & 0 & 0 & -1 & 0 & 0 \\
1 & 0 & 0 & 1 & 0 & 0 \\
0 & 0 & 0 & 0 & -1 & 0 \\
0 & -1 & 0 & 0 & 0 & 0
\end{array}\right]\text{ and }
N^\dagger=\frac{1}{7} \left[\begin{array}{rrrrrrr}
-2 & 2 & -2 & 5 & 5 & 2 & -2 \\
1 & -1 & 1 & 1 & 1 & -1 & -6 \\
3 & 4 & 3 & -4 & -4 & -3 & 3 \\
1 & -1 & 1 & -6 & 1 & -1 & 1 \\
-1 & 1 & -1 & -1 & -1 & -6 & -1 \\
2 & -2 & -5 & 2 & 2 & 5 & 2
\end{array}\right]. $$
$$(N^\dagger)_{54}=-\frac{\sgn(P_{e_5-4})w(e_5)}{7}\sgn(e_5)|T_t(e_5)|=-\frac{1\cdot (-1)}{7}\cdot (-1)\cdot 1=-\frac{1}{7}$$ since $4\in T_h(e_5)$ and $w(e_5)=-1$.

\end{example}

\begin{remark}
If $\Gamma$ be a signed tree with all positive edges, then \eqref{H} boils down to $$n_{i,j}^\dagger=
 \frac{1}{n}\begin{cases}
 |T_h(e_i)| & \text{ if } j \in T_t(e_i)\\
 -|T_t(e_i)| & \text{ if } j \in T_h(e_i),
 \end{cases}$$  i.e., our formula matches with  the known formula given in \cite{bap}.

If $\Gamma$ be a signed tree with all negative edges having both ends equal to $1$, then
$$n_{i,j}^\dagger=
 \frac{(-1)^{d(e_i,j)}}{n}\begin{cases}
 |T_h(e_i)| & \text{ if } j \in T_t(e_i)\\
 |T_t(e_i)| & \text{ if } j \in T_h(e_i).
 \end{cases}$$ In other words, formula  \eqref{H} is  matching the one given in \cite{AOM}.
\end{remark}

We recall the notation introduced in \cite{HM2}. Let $G$ be a uniyclic graph with the unique cycle $C$. Then
$G\setminus e_i[C]$ is the connected component of $G\setminus e_i$ that contains $C$ if $e_i\notin C$ and 
$G\setminus e_i$ otherwise. Similarly, 
$G\setminus e_i(C)$ is the
connected component of $G\setminus e_i$ that does not contain  $C$  if $e_i\in C$ and 
empty graph, otherwise.  It $e_i\notin C$, then by $G_h(e_i)$ we denote the component of $G$ that contains the head of $e_i.$

As pointed out in  Theorem \ref{ps}  incidence matrix of any unbalanced signed graph is nonsingular. We next give combinatorial expressions for the entries of its inverse.

\begin{theorem}\label{uni}
Let $\Gamma=(G,\sigma)$ be an unbalanced unicyclic signed graph on $n$ vertices $1,2,\ldots,n$ and edges $e_1,e_2,\ldots,e_n$ with the cycle $C$ and the incidence matrix $N$. Then $N$ is invertible and 
\begin{equation}\label{unicyclic N inv}%\nonumber
(N^{-1})_{ij}=w(e_i) \begin{cases}
\frac{1}{2}\sgn\left(P_{e_i^t-j}\right) & \text{ if } e_i \in C\\
0 & \text{ if } e_i \not \in C \text{ and } j \in G\setminus e_i [C]\\
-\sgn(P_{e_i-j}) & \text{ if } e_i \in E^+(\Gamma)\setminus E(C) \text{ and } j \not \in G\setminus e_i [C] \text{ and } j\in G_h(e_i)\\
\sgn(P_{e_i-j}) & \text{ otherwise, }
\end{cases} 
\end{equation}
where $P_{e_i^t-j}$ denotes the path  between the tail of $e_i\in C$ and $j$ in the tree $G\setminus e_i$ and $w(e_i)=-1$ if $e_i$ is a negative edge with both negative ends and $1$ otherwise.
 \end{theorem}

\begin{proof}
Again we prove that the result holds for signed graphs in which all negative edges have both positive ends. The final formula follows similarly as in the proof Theorem \ref{tree}.

Since $G$ is unbalanced, $\rank\, N=n$ and $N$ is invertible. Let  $A=[a_{ij}]$ where $a_{ij}$ is defined as in (\ref{unicyclic N inv}). It remains to show that $AN=I_{n}$.    

\[(AN)_{ij}=\begin{cases}
a_{iu}+a_{iv} & \text{ if }  e_j=\{u,v\} \in E^-(\Gamma)\\
a_{iu}-a_{iv} & \text{ if }  e_j=(u,v) \in E^+(\Gamma)\\
-a_{iu}+a_{iv} & \text{ if }  e_j=(v,u) \in E^+(\Gamma).
\end{cases} \]

First suppose $i=j$.  If $e_i=e_j \in C$, then ($C$ being unbalanced)
\[(AN)_{ij}=\begin{cases}
\frac{1}{2}\sgn\left(P_{e_{i}^t-u}\right)+\frac{1}{2}\sgn\left(P_{e_{i}^t-v}\right)=\frac{1}{2}+\frac{1}{2}=1 & \text{ if }  e_j=\{u,v\} \in E^-(\Gamma)\\
\frac{1}{2}\sgn\left(P_{e_{i}^t-u}\right)-\frac{1}{2}\sgn\left(P_{e_{i}^t-v}\right)=\frac{1}{2}-(-\frac{1}{2})=1 & \text{ if }  e_j=(u,v) \in E^+(\Gamma)\\
-\frac{1}{2}\sgn\left(P_{e_{i}^t-u}\right)+\frac{1}{2}\sgn\left(P_{e_{i}^t-v}\right)=-(-\frac{1}{2})+\frac{1}{2}=1 & \text{ if }  e_j=(v,u) \in E^+(\Gamma).
\end{cases} \]

If $e_i=e_j \notin C$, then WLOG let $u \in G\setminus e_i [C]$ and $v \notin G\setminus e_i [C]$. Then $a_{iu}=0$,  $a_{iv}=\pm \sgn(P_{e_i-v})$, and
\[(AN)_{ij}=\begin{cases}
0-(-\sgn(P_{e_i-v}))=1 & \text{ if }  e_j=(u,v) \in E^+(\Gamma)\\
-0+\sgn(P_{e_i-v}) =1 & \text{ otherwise}.
\end{cases} \]

\bigskip
\noindent Now suppose $i\neq j$ and $e_j=\{u,v\}$. Based on the positions of $e_i, e_j$ with respect to $C$ we separate the following cases.\\

\noindent{\it Case 1.} $e_i\notin C$ and $e_j\in C.$

Since $e_i\notin C$ and $u,v\in G\setminus e_i [C]$, $a_{iu}=a_{iv}=0$ and $(AN)_{ij}=0$.\\

\noindent{\it Case 2.} $e_i,e_j\in C.$

If $e_j\in E^-(\Gamma)$, then $\sgn\left(P_{e_{i}^t-u}\right)$ and $\sgn\left(P_{e_{i}^t-v}\right)$ are of the  opposite signs ($C$ being unbalanced) and 
\[(AN)_{ij}=a_{iu}+a_{iv}
=\frac{1}{2}\sgn\left(P_{e_{i}^t-u}\right)
+\frac{1}{2}\sgn\left(P_{e_{i}^t-v}\right)=0. \]

If $e_j\in E^+(\Gamma)$, then $\sgn\left(P_{e_{i}^t-u}\right)$ and $\sgn\left(P_{e_{i}^t-v}\right)$ have the same sign and 
\[(AN)_{ij}=\pm(a_{iu}-a_{iv})
=\pm\left(\frac{1}{2}\sgn\left(P_{e_{i}^t-u}\right)
-\frac{1}{2}\sgn\left(P_{e_{i}^t-v}\right) \right)=0. \]

\bigskip
\noindent{\it Case 3.} $e_i,e_j\notin C.$

WLOG let $u$ be the vertex on $e_j$ that is closest to $e_i$. Then $\sgn(P_{e_i-v})=\sigma(e_j)\sgn(P_{e_i-u})$. 

If $e_j\in E^-(\Gamma)$, then $\sgn(P_{e_i-u})=-\sgn(P_{e_i-v})$ and
\[(AN)_{ij}=a_{iu}+a_{iv}
=\sgn(P_{e_i-u})+\sgn(P_{e_i-v})=0. \]

If $e_j\in E^+(\Gamma)$, then $\sgn(P_{e_i-u})=\sgn(P_{e_i-v})$  and 
\[(AN)_{ij}=\pm(a_{iu}-a_{iv})
=\pm\left( \sgn(P_{e_i-u})-\sgn(P_{e_i-v}) \right)=0. \]

\bigskip
\noindent{\it Case 4.} $e_i\in C$ and $e_j\notin C$

WLOG let $u$ be the vertex on $e_j$ that is closest to the tail of $e_i$ in the tree $G\setminus e_i$. Then $\sgn\left(P_{e_i^t-v}\right)=\sigma(e_j) \sgn\left(P_{e_i^t-u}\right)$.

If $e_j\in E^-(\Gamma)$, then $\sgn\left(P_{e_i^t-u}\right)$ and $\sgn\left(P_{e_i^t-v}\right)$ have opposite signs and 
\[(AN)_{ij}=a_{iu}+a_{iv}
=\frac{1}{2}\sgn\left(P_{e_i^t-u}\right)+\frac{1}{2}\sgn\left(P_{e_i^t-v}\right)=0. \]

If $e_j\in E^+(\Gamma)$, then $\sgn\left(P_{e_i^t-u}\right)$ and $\sgn\left(P_{e_i^t-v}\right)$ have the same sign and 
\[(AN)_{ij}=\pm(a_{iu}-a_{iv})
=\pm\left( \frac{1}{2}\sgn\left(P_{e_i^t-u}\right)-\frac{1}{2}\sgn\left(P_{e_i^t-v}\right) \right)=0. \]
This completes the proof.
\end{proof}

\begin{example}
For the signed unbalanced graph depicted in Figure \ref{fig:even unicyclic},  the  results of Theorem  \ref{uni} are exemplified  in computations of $(N^{-1})_{74}$ and $(N^{-1})_{26}.$

$$(N^{-1})_{74}=\frac{1}{2}w(e_7)\sgn(P_{e_7^t-4})=\frac{1}{2}\cdot 1\cdot(-1)=-\frac{1}{2}$$
$$(N^{-1})_{26}=w(e_2)\sgn(P_{e_2-6})=(-1)\cdot 1=-1.$$
\medskip

\begin{figure}[h]
	\centering
	\def\NodeSC{0.75}
	\def\Hdis{6.0}
	\def\Vdis{3.0}
	\begin{tikzpicture}[line width=0.25mm, baseline, scale=0.9]
		\node[whiteDot,scale=\NodeSC] (A6) at (0,0) {\bf $6$};
		\path (A6) +(90:\Vdis) node[whiteDot,scale=\NodeSC] (A8) {\bf $8$};
		\path (A6) +(0:\Hdis) node[whiteDot,scale=\NodeSC] (A9) {\bf $9$};
		\path (A8) -- (A9) node[midway,whiteDot,scale=\NodeSC] (A4) {\bf $4$};
		\path (A8) +(0:\Hdis) node[whiteDot,scale=\NodeSC] (A7) {\bf $7$};
		\path (A9) +(0:1.5*\Hdis) node[whiteDot,scale=\NodeSC] (A1) {\bf $1$};
		\path (A7) +(0:1.5*\Hdis) node[whiteDot,scale=\NodeSC] (A3) {\bf $3$};
		\path (A4) +(0:\Hdis) node[whiteDot,scale=\NodeSC] (A5) {\bf $5$};
		\path (A5) +(0:0.5*\Hdis) node[whiteDot,scale=\NodeSC] (A2) {\bf $2$};
	    \begin{pgfonlayer}{background}
	    	\draw[line width=\lineWidth,<->] (A6) to node[midway, below] {\emph{$-e_2$}} (A4);
			\node[above, rotate=25] at ($(A6)!.45!(A4)$) {\small{$w(e_2) = -1$}};
	        \draw[line width=\lineWidth,>->] (A4) to node[midway, above] {\emph{$+e_1$}} (A8);
	        \draw[line width=\lineWidth,>-<] (A4) to node[midway, above] {\emph{$-e_9$}} (A7);
	        \draw[line width=\lineWidth,<-<] (A4) to node[midway, below] {\emph{$+e_7$}} (A9);
	        \draw[line width=\lineWidth,>-<] (A7) to node[midway, above] {\emph{$-e_6$}} (A5);
	        \draw[line width=\lineWidth,>-<] (A9) to node[midway, below] {\emph{$-e_8$}} (A5);
	        \draw[line width=\lineWidth,>-<] (A5) to node[midway, above] {\emph{$-e_4$}} (A2);
	        \draw[line width=\lineWidth,>-<] (A2) to node[midway, below] {\emph{$-e_3$}} (A1);
	        \draw[line width=\lineWidth,>->] (A2) to node[midway, above] {\emph{$+e_5$}} (A3);
	      \end{pgfonlayer}
	  \end{tikzpicture}
\vspace{1cm}
\caption{An unbalanced even unicyclic graph $\Gamma_1$.}
\label{fig:even unicyclic}
\end{figure}
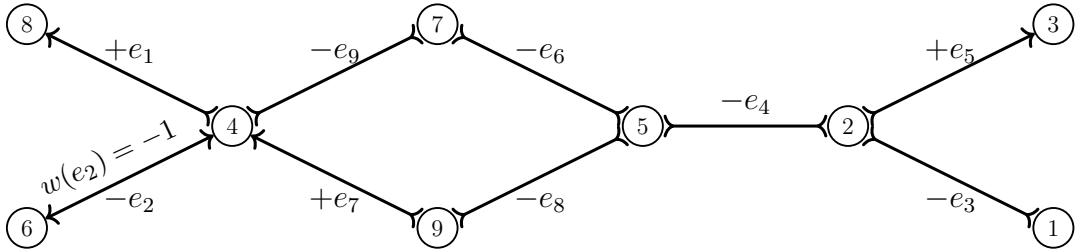

In particular, 
\[\scriptsize N=\left[\begin{array}{rrrrrrrrr}
0 & 0 & 1 & 0 & 0 & 0 & 0 & 0 & 0 \\
0 & 0 & 1 & 1 & 1 & 0 & 0 & 0 & 0 \\
0 & 0 & 0 & 0 & -1 & 0 & 0 & 0 & 0 \\
1 & -1 & 0 & 0 & 0 & 0 & -1 & 0 & 1 \\
0 & 0 & 0 & 1 & 0 & 1 & 0 & 1 & 0 \\
0 & -1 & 0 & 0 & 0 & 0 & 0 & 0 & 0 \\
0 & 0 & 0 & 0 & 0 & 1 & 0 & 0 & 1 \\
-1 & 0 & 0 &0 & 0 & 0 & 0 & 0 & 0 \\
0 & 0 & 0 & 0 & 0 & 0 & 1 & 1 & 0
\end{array}\right],\;
N^{-1}=\frac{1}{2}
\left[\begin{array}{rrrrrrrrr}
0 & 0 & 0 & 0 & 0 & 0 & 0 & -2 & 0 \\
0 & 0 & 0 & 0 & 0 & -2 & 0 & 0 & 0 \\
2 & 0 & 0 & 0 & 0 & 0 & 0 & 0 & 0 \\
-2 & 2 & 2 & 0 & 0 & 0 & 0 & 0 & 0 \\
0 & 0 & -2 & 0 & 0 & 0 & 0 & 0 & 0 \\
1 & -1 & -1 & -1 & 1 & 1 & 1 & -1 & -1 \\
-1 & 1 & 1 & -1 & -1 & 1 & 1 & -1 & 1 \\
1 & -1 & -1 & 1 & 1 & -1 & -1 & 1 & 1 \\
-1 & 1 & 1 & 1 & -1 & -1 & 1 & 1 & 1
\end{array}\right]
\]

\end{example}

\smallskip
\begin{remark}
If  $\Gamma$ is  an unbalanced odd unicyclic signed graph with all negative edges, with both positive ends, then the general  formula \eqref{unicyclic N inv} is matching the known formula given in \cite{HM2}:
$$(N^{-1})_{ij}= \begin{cases}
\frac{1}{2}(-1)^{d(e_i,j)} & \text{ if } e_i \in C\\
0 & \text{ if } e_i \not \in C \text{ and } j \in V(G)\setminus e_i [C]\\(-1)^{d(e_i,j)}& \text{ if } e_i \notin E(C) \text{ and } j \not \in V(G)\setminus e_i [C],
\end{cases}$$
due to $w(e_i)=1$ for any $e_i\in E(\Gamma)$.

\end{remark}

\section{Moore-Penrose inverse of the incidence matrix of a signed graph}\label{sec:4}

\noindent 
In this section we implement results of previous two, in order to obtain expression for the Moore-Penrose inverse of the incidence matrix of any connected signed graph.

\begin{theorem}\label{gen}
Let $\Gamma=(G,\sigma)$ be a connected signed graph on $n\geq 2$ vertices $1,2,\ldots,n$ and $m$ edges $e_1,e_2,\ldots,e_m$ with incidence matrix $N$, $w(e_i)$ being $-1$ for negative edges with both negative ends and $1$, otherwise. 
\begin{enumerate}
\item[(a)] If $\Gamma$ is balanced, then 
\begin{equation}
(N^\dagger)_{ij}= \frac{1}{\tau(\Gamma)} \sum_{\substack{T\in \mathcal{S}(\Gamma)\\e_i\in T}}  (N_T^\dagger)_{ij},
\end{equation}
where $\mathcal{S}(\Gamma)$ is the set of all spanning trees $T$ of $\Gamma$ and when $e_i\in T$,
\begin{equation}
(N_T^\dagger)_{ij}=
 \frac{w(e_i)\sgn(P_{e_i-j})}{n}\begin{cases}
 |T_h(e_i)| & \text{ if } j \in T_t(e_i)\\
 -\sigma(e_i)|T_t(e_i)| & \text{ if } j \in T_h(e_i)
 \end{cases} 
\end{equation}
($P_{e_i-j}$ being in $T$).

\item[(b)] If $\Gamma$ is unbalanced, then
\begin{equation}
(N^{-1})_{ij}= \frac{1}{\sum_{H\in \mathcal{TU}_n(\Gamma)} 4^{c(H)}}\sum_{\substack{H\in \mathcal{TU}_n(\Gamma)\\e_i\in H}} 4^{c(H)} (N_H^{-1})_{ij},
\end{equation}
where $\mathcal{TU}_n(\Gamma)$ is the set of all spanning $TU$-subgraphs $H$ of $\Gamma$ with $n$ edges consisting of $c(H)$ unbalanced unicyclic graphs $U^H_1,U^H_2,\ldots,U^H_{c(H)}$ with cycles $C^H_1,C^H_2,\ldots,C^H_{c(H)}$ respectively  and when $e_i\in U^H_r$ for some $r$,
\begin{equation}\label{Unbalanced G N inverse}
(N_H^{-1})_{ij}= \begin{cases}
\frac{1}{2}\sgn\left(P_{e_i^t-j}\right) & \text{ if } e_i \in C^H_r\\
0 & \text{ if } e_i \not \in C_r \text{ and } j \in U^H_r\setminus e_i [C^H_r]\\
-\sgn(P_{e_i-j}) & \text{ if } e_i \in E^+(U^H_r)\setminus E(C^H_r) \text{ and } j \not \in U^H_r\setminus e_i [C^H_r] \text{ and } j\in (U^H_r)_h(e_i)\\
\sgn(P_{e_i-j}) & \text{ otherwise, }
\end{cases} 
\end{equation}
where $P_{e_i^t-j}$ denotes the path  between the tail of $e_i\in C^H_r$ and $j$ in the tree $U_r^H\setminus e_i$.
\end{enumerate}

\end{theorem}

\begin{proof}
The formulas follow by \eqref{MPdet} and Theorems \ref{vol}, \ref{tree}, \ref{uni} taking into account that for 
 a $TU$-spanning subgraph $H$ of $\Gamma$ with $c(H)$ unbalanced unicyclic graphs, $U^H_1, \ldots, U^H_{c(H)}$, with unique odd cycles $C^H_1,\ldots, C^H_{c(H)}$ respectively,   the incidence matrix $N_H$ is of the form
$$N_H=\begin{pmatrix} N_{U^H_1}\oplus\cdots\oplus N_{U^H_{c(H)}}& O
\end{pmatrix}P,$$ where
$N_{U^H_\ell}$ denotes  the incidence matrix of $U^H_\ell$, $1\leq \ell\leq c(H)$ and $P$ is the corresponding permutation matrix. Then  $$(N_H)^{\dagger}=P^\intercal\begin{pmatrix} (N_{U^H_1})^{-1}\oplus\cdots\oplus (N_{U^H_{c(H)}})^{-1}\\O
\end{pmatrix},$$  taking into account that $P^\intercal=P^{-1}$.
\end{proof}

Next, we provide the explicit formula for the Moore-Penrose inverse of balanced unicyclic graph.

\begin{theorem}\label{unib}
Let $\Gamma=(G,\sigma)$ be a connected balanced unicyclic graph on $n$ vertices $1,2,\ldots,n$ with incidence matrix $N$, the unique cycle $C$ with  $E(C)=\{e_1,\ldots, e_{|C|}\}$, $T_k=G\setminus e_k$, $1\leq k\leq |C|$ and $w(e_i)$ being $-1$ for negative edges with both negative ends and $1$, otherwise.  Then the following holds.
\begin{itemize}
    \item[(a)] If
$e_i\notin C$ and $j\in G\setminus e_i[C]$, then either $j\in T_h^{T_k}(e_i)$ for all $k\in\{1, \ldots, |C|\}$ or $j\in T_t^{T_k}(e_i)$ for all $k\in\{1, \ldots, |C|\}$.
\item[(b)] $$(N^\dagger)_{ij}=\frac{w(e_i)\sgn(P_{e_i-j})}{n|C|}\cdot\begin{cases}
|C||G\setminus e_i(C)|\,\,\,\mbox{if}\,\,\, e_i\notin E(C), j\in G\setminus e_i[C], j\in  T_t^{T_k}(e_i) \mbox{ for some $T_k$,}\\
-\sgn(e_i)|C||G\setminus e_i[C]|\,\,\,\mbox{if}\,\,\, e_i\notin E(C),  j\in G\setminus e_i[C], j\in  T_h^{T_k}(e_i)  \mbox{ for some $T_k$,}\\
\sum_{k=1\atop
e_i\in T_k}^{|C|}\phi_{T_k}(e_i,j)\,\,\,\mbox{if}\,\,\, e_i\in E(C), j\in V(G),
\end{cases}
$$
\end{itemize}
where $$\phi_{T}(e_i,j)=
\begin{cases}
|T_h(e_i)|\mbox{ if } j\in T_t(e_i)\\
-\sigma(e_i)|T_t(e_i)|\mbox{ if } j\in T_h(e_i).
\end{cases}$$ 
\end{theorem}

\begin{proof}
Since $\Gamma$ is connected balanced unicyclic graph,  the number of spanning trees is  $\tau(G)=|C|$ and the spanning trees are $T_k, 1\leq k\leq |C|$. Let $e_i\notin C,$ and $j\in G\setminus e_i[C]$. Then for all $k$, $T_k\setminus e_i$ contains  $j$ either in the tail or in  the head. In addition,  $|T^{T_k}_t(e_i)|=|G\setminus e_i(C)|,$ while $|T^{T_k}_h(e_i)|=|G\setminus e_i[C)].$

By Theorem \ref{gen} (a)
$$(N^{\dagger})_{ij}=\frac{1}{|C|}\sum_{k=1}^{|C|} (N^\dagger_{T_k})_{ij},$$
where
$$(N^\dagger_{T_k})_{ij}=\frac{w(e_i)\sgn(P_{e_i-j})}{n}\begin{cases}
    |G\setminus e_i(C)|& \mbox{ if $ e_i\notin E(C), j\in G\setminus e_i[C], j\in T_t^{T_k}(e_i)$ }\\
    |G\setminus e_i[C]|& \mbox{ if $ e_i\notin E(C), j\in G\setminus e_i[C], j\in  T_h^{T_k}(e_i)$ }\\
    |T^{T_k}_h(e_i)|&\mbox{ if $ e_i\in C, j\in T^{T_k}_t(e_i)$ }\\
-\sigma(e_i) |T^{T_k}_t(e_i)|&\mbox{ if $ e_i\in C, j\in T^{T_k}_h(e_i)$. }
\end{cases}$$
Hence, 

$$\sum_{k=1}^{|C|} (N^\dagger_{T_k})_{ij}=\frac{w(e_i)\sgn(P_{e_i-j})}{n}\begin{cases}
   |C| |G\setminus e_i(C)|& \mbox{ if $ e_i\notin E(C), j\in G\setminus e_i[C], j\in T_t^{T_k}(e_i)$ }\\
   |C| |G\setminus e_i[C]|& \mbox{ if $ e_i\notin E(C), j\in G\setminus e_i[C], j\in  T_h^{T_k}(e_i)$ }\\
   \sum_{k=1\atop
e_i\in T_k}^{|C|}\phi_{T_k}(e_i,j)&\mbox{ if}\,\,\, e_i\in E(C), j\in V(G).
\end{cases}
$$
Finally, devision by $|C|$ in the previous equality leads to $(N^{\dagger})_{ij}$, i.e., to the  expression given in (b).
\end{proof}

\begin{example}
For the unbalanced graph, $\Gamma_2$ depicted in Figure \ref{fig:bi}, spanning $TU$-subgraphs on $9$ edges are enclosed in Figure \ref{tu_2}.
\tikzset{middlearrow/.style={
        decoration={markings,
            mark= at position 0.5 with {\arrow[scale=2]{#1}} ,
        },
        postaction={decorate}
    }
}
\tikzset{doublearrow/.style={
        decoration={markings,
            mark= at position 0.5 with {\arrow[scale=2]{#1}} ,
            mark= at position 0.55 with {\arrow[scale=2]{#1}} ,
        },
        postaction={decorate}
    }
}

\begin{figure}[h]
	\centering
	\def\NodeSC{0.75}
	\def\Hdis{6.0}
	\def\Vdis{3.0}
	\begin{tikzpicture}[line width=0.25mm, baseline, scale=0.9]
		\node[whiteDot,scale=\NodeSC] (A6) at (0,0) {\bf $6$};
		\path (A6) +(90:\Vdis) node[whiteDot,scale=\NodeSC] (A8) {\bf $8$};
		\path (A6) +(0:\Hdis) node[whiteDot,scale=\NodeSC] (A9) {\bf $9$};
		\path (A8) -- (A9) node[midway,whiteDot,scale=\NodeSC] (A4) {\bf $4$};
		\path (A8) +(0:\Hdis) node[whiteDot,scale=\NodeSC] (A7) {\bf $7$};
		\path (A9) +(0:1.5*\Hdis) node[whiteDot,scale=\NodeSC] (A1) {\bf $1$};
		\path (A7) +(0:1.5*\Hdis) node[whiteDot,scale=\NodeSC] (A3) {\bf $3$};
		\path (A4) +(0:\Hdis) node[whiteDot,scale=\NodeSC] (A5) {\bf $5$};
		\path (A5) +(0:0.5*\Hdis) node[whiteDot,scale=\NodeSC] (A2) {\bf $2$};
	    \begin{pgfonlayer}{background}
	    	\draw[line width=\lineWidth,>-<] (A6) to node[midway, below] {\emph{$-e_2$}} (A4);
	        \draw[line width=\lineWidth,>->] (A4) to node[midway, above] {\emph{$+e_1$}} (A8);
	        \draw[line width=\lineWidth,>-<] (A4) to node[midway, above] {\emph{$-e_9$}} (A7);
	        \draw[line width=\lineWidth,<-<] (A4) to node[midway, below] {\emph{$+e_7$}} (A9);
	        \draw[line width=\lineWidth,>-<] (A7) to node[midway, above] {\emph{$-e_6$}} (A5);
	        \draw[line width=\lineWidth,>-<] (A9) to node[midway, below] {\emph{$-e_8$}} (A5);
	        \draw[line width=\lineWidth,>-<] (A5) to node[midway, above] {\emph{$-e_4$}} (A2);
	        \draw[line width=\lineWidth,>-<] (A2) to node[midway, below] {\emph{$-e_3$}} (A1);
	        \draw[line width=\lineWidth,>->] (A2) to node[midway, above] {\emph{$+e_5$}} (A3);
	        \draw[line width=\lineWidth,>-<] (A1) to node[midway, right] {\emph{$-e_{10}$}} (A3);
	      \end{pgfonlayer}
	  \end{tikzpicture}
\vspace{1cm}
	\caption{An unbalanced  bicyclic graph $\Gamma_2$.}
	\label{fig:bi}
\end{figure}
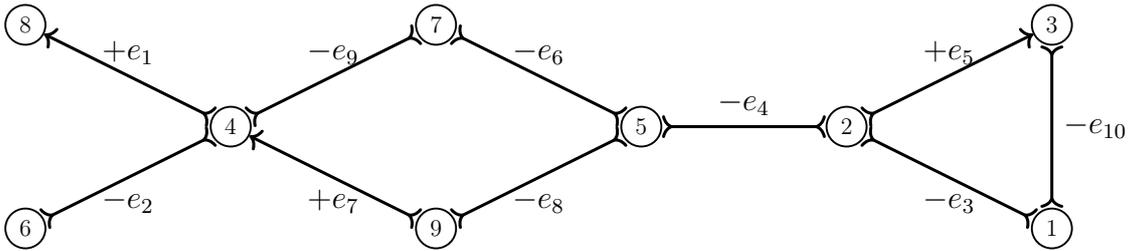

	\begin{figure}[h]
	\centering
		\def\dis{3cm}
	    \def\NodeSC{0.65}
	    \def\Hdis{4.0}
	    \def\Vdis{3.0}
	    \begin{tikzpicture}[name=G_11, line width=0.25mm, baseline, scale=0.55, descr/.style={sloped,below}]
	      \node[whiteDot,scale=\NodeSC] (A6) at (0,0) {\bf $6$};
		  	\path (A6) +(90:\Vdis) node[whiteDot,scale=\NodeSC] (A8) {\bf $8$};
		  	\path (A6) +(0:\Hdis) node[whiteDot,scale=\NodeSC] (A9) {\bf $9$};
		  	\path (A8) +(0:\Hdis) node[whiteDot,scale=\NodeSC] (A7) {\bf $7$};
		  	\path (A7) +(0:1.5*\Hdis) node[whiteDot,scale=\NodeSC] (A3) {\bf $3$};
		  	\path (A9) +(0:1.5*\Hdis) node[whiteDot,scale=\NodeSC] (A1) {\bf $1$};
			\path (A6) +(0:0.5*\Hdis) node[] (A44) {};
		  	\path (A44) +(90:0.5*\Vdis) node[whiteDot,scale=\NodeSC] (A4) {\bf $4$};
		  	\path (A4) +(0:\Hdis) node[whiteDot,scale=\NodeSC] (A5) {\bf $5$};
		  	\path (A5) +(0:0.5*\Hdis) node[whiteDot,scale=\NodeSC] (A2) {\bf $2$};
	      \begin{pgfonlayer}{background}
	        \draw[line width=\lineWidth,>-<] (A4) to node {} (A6);
	        \draw[line width=\lineWidth,>-<] (A4) to node {} (A7);
	        \draw[line width=\lineWidth,>->] (A4) to node {} (A8);
	        \draw[line width=\lineWidth,<-<] (A4) to node {} (A9);
	        \draw[line width=\lineWidth,>-<] (A5) to node {} (A7);
	        \draw[line width=\lineWidth,>-<] (A5) to node {} (A9);
	        \draw[line width=\lineWidth,>-<] (A5) to node {} (A2);
	        \draw[line width=\lineWidth,>-<] (A2) to node {} (A1);
	        %\draw[line width=\lineWidth,>->] (A2) to node {} (A3);
	        \draw[line width=\lineWidth,>-<] (A1) to node {} (A3);
	      \end{pgfonlayer}
		  \node[anchor=north] at (current bounding box.south){$G_{11}$};
	    \end{tikzpicture}
		\hspace{0.5cm}
	    \begin{tikzpicture}[name=G_12, line width=0.25mm, baseline, scale=0.65, descr/.style={sloped,below}]
	      \node[whiteDot,scale=\NodeSC] (A6) at (0,0) {\bf $6$};
		  	\path (A6) +(90:\Vdis) node[whiteDot,scale=\NodeSC] (A8) {\bf $8$};
		  	\path (A6) +(0:\Hdis) node[whiteDot,scale=\NodeSC] (A9) {\bf $9$};
		  	\path (A8) +(0:\Hdis) node[whiteDot,scale=\NodeSC] (A7) {\bf $7$};
		  	\path (A7) +(0:1.5*\Hdis) node[whiteDot,scale=\NodeSC] (A3) {\bf $3$};
		  	\path (A9) +(0:1.5*\Hdis) node[whiteDot,scale=\NodeSC] (A1) {\bf $1$};
			\path (A6) +(0:0.5*\Hdis) node[] (A44) {};
		  	\path (A44) +(90:0.5*\Vdis) node[whiteDot,scale=\NodeSC] (A4) {\bf $4$};
		  	\path (A4) +(0:\Hdis) node[whiteDot,scale=\NodeSC] (A5) {\bf $5$};
		  	\path (A5) +(0:0.5*\Hdis) node[whiteDot,scale=\NodeSC] (A2) {\bf $2$};
	      \begin{pgfonlayer}{background}
	        \draw[line width=\lineWidth,>-<] (A4) to node {} (A6);
	        \draw[line width=\lineWidth,>-<] (A4) to node {} (A7);
	        \draw[line width=\lineWidth,>->] (A4) to node {} (A8);
	        \draw[line width=\lineWidth,<-<] (A4) to node {} (A9);
	        \draw[line width=\lineWidth,>-<] (A5) to node {} (A7);
	        \draw[line width=\lineWidth,>-<] (A5) to node {} (A9);
	        \draw[line width=\lineWidth,>-<] (A5) to node {} (A2);
	        \draw[line width=\lineWidth,>-<] (A2) to node {} (A1);
	        \draw[line width=\lineWidth,>->] (A2) to node {} (A3);
	        %\draw[line width=\lineWidth,>-<] (A1) to node {} (A3);
	      \end{pgfonlayer}
		  \node[anchor=north] at (current bounding box.south){$G_{12}$};
	    \end{tikzpicture}
		\hspace{0.5cm}
	    \begin{tikzpicture}[name=G_13, line width=0.25mm, baseline, scale=0.55, descr/.style={sloped,below}]
	      \node[whiteDot,scale=\NodeSC] (A6) at (0,0) {\bf $6$};
		  	\path (A6) +(90:\Vdis) node[whiteDot,scale=\NodeSC] (A8) {\bf $8$};
		  	\path (A6) +(0:\Hdis) node[whiteDot,scale=\NodeSC] (A9) {\bf $9$};
		  	\path (A8) +(0:\Hdis) node[whiteDot,scale=\NodeSC] (A7) {\bf $7$};
		  	\path (A7) +(0:1.5*\Hdis) node[whiteDot,scale=\NodeSC] (A3) {\bf $3$};
		  	\path (A9) +(0:1.5*\Hdis) node[whiteDot,scale=\NodeSC] (A1) {\bf $1$};
			\path (A6) +(0:0.5*\Hdis) node[] (A44) {};
		  	\path (A44) +(90:0.5*\Vdis) node[whiteDot,scale=\NodeSC] (A4) {\bf $4$};
		  	\path (A4) +(0:\Hdis) node[whiteDot,scale=\NodeSC] (A5) {\bf $5$};
		  	\path (A5) +(0:0.5*\Hdis) node[whiteDot,scale=\NodeSC] (A2) {\bf $2$};
	      \begin{pgfonlayer}{background}
	        \draw[line width=\lineWidth,>-<] (A4) to node {} (A6);
	        \draw[line width=\lineWidth,>-<] (A4) to node {} (A7);
	        \draw[line width=\lineWidth,>->] (A4) to node {} (A8);
	        \draw[line width=\lineWidth,<-<] (A4) to node {} (A9);
	        \draw[line width=\lineWidth,>-<] (A5) to node {} (A7);
	        \draw[line width=\lineWidth,>-<] (A5) to node {} (A9);
	        \draw[line width=\lineWidth,>-<] (A5) to node {} (A2);
	        %\draw[line width=\lineWidth,>-<] (A2) to node {} (A1);
	        \draw[line width=\lineWidth,>->] (A2) to node {} (A3);
	        \draw[line width=\lineWidth,>-<] (A1) to node {} (A3);
	      \end{pgfonlayer}
		  \node[anchor=north] at (current bounding box.south){$G_{13}$};
	    \end{tikzpicture}
		\vspace{0.3cm}
		\caption{Spanning $TU$-subgraphs of $\Gamma_2$ on $9$ edges.}\label{tu_2}
	\end{figure}
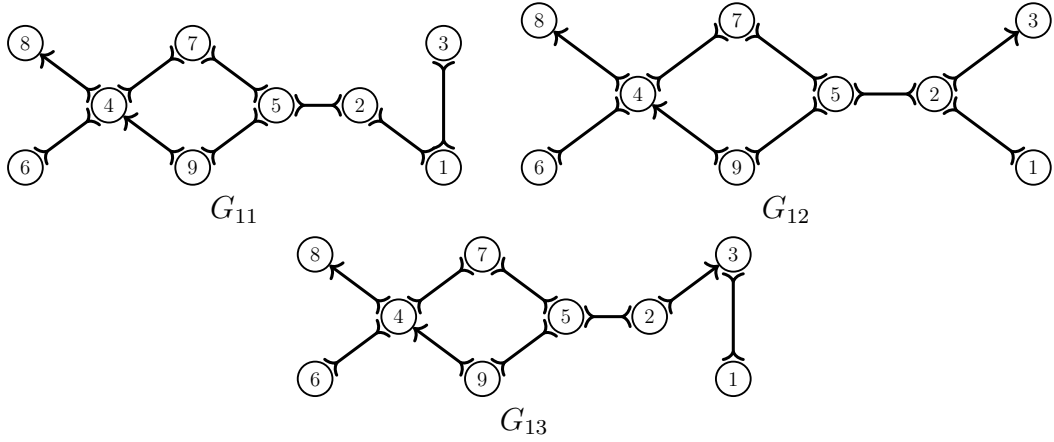

 In this case $\vol^2 N=4+4+4=12.$
	Then,  $$(N^\dagger)_{53}=\frac{1}{12}(4\cdot(-\sgn_{G_{12}}(P_{e_5-3}))+4(-\sgn_{G_{13}}(P_{e_5-3})))=-\frac{8}{12}=-\frac{2}{3}.$$
\end{example}

\begin{example} In this example we consider unbalanced bicyclic graphs $\Gamma_3$ with two negative cycles.
	\begin{figure}[h]
	\centering
		\def\dis{3cm}
	    \def\NodeSC{0.75}
	    \def\Hdis{4.5}
	    \def\Vdis{3.5}
	    \begin{tikzpicture}[line width=0.25mm, baseline, scale=1.0, descr/.style={sloped,below}]
	      \node[whiteDot,scale=\NodeSC] (A6) at (0,0) {\bf $6$};
		  	\path (A6) +(90:\Vdis) node[whiteDot,scale=\NodeSC] (A8) {\bf $8$};
		  	\path (A6) +(0:\Hdis) node[whiteDot,scale=\NodeSC] (A9) {\bf $9$};
		  	\path (A8) +(0:\Hdis) node[whiteDot,scale=\NodeSC] (A7) {\bf $7$};
		  	\path (A7) +(0:1.5*\Hdis) node[whiteDot,scale=\NodeSC] (A3) {\bf $3$};
		  	\path (A9) +(0:1.5*\Hdis) node[whiteDot,scale=\NodeSC] (A1) {\bf $1$};
			\path (A6) +(0:0.5*\Hdis) node[] (A44) {};
		  	\path (A44) +(90:0.5*\Vdis) node[whiteDot,scale=\NodeSC] (A4) {\bf $4$};
		  	\path (A4) +(0:\Hdis) node[whiteDot,scale=\NodeSC] (A5) {\bf $5$};
		  	\path (A5) +(0:0.5*\Hdis) node[whiteDot,scale=\NodeSC] (A2) {\bf $2$};
	      \begin{pgfonlayer}{background}
	        \draw[line width=\lineWidth,>-<] (A4) to node[midway, below] {\emph{$-e_2$}} (A6);
	        \draw[line width=\lineWidth,>-<] (A4) to node[midway, below] {\emph{$-e_9$}} (A7);
	        \draw[line width=\lineWidth,>->] (A4) to node[midway, below] {\emph{$+e_1$}} (A8);
	        \draw[line width=\lineWidth,>->] (A4) to node[midway, below] {\emph{$+e_7$}} (A9);
	        \draw[line width=\lineWidth,>-<] (A5) to node[midway, below] {\emph{$-e_6$}} (A7);
	        \draw[line width=\lineWidth,>-<] (A5) to node[midway, below] {\emph{$-e_8$}} (A9);
	        \draw[line width=\lineWidth,>-<] (A5) to node[midway, above] {\emph{$-e_4$}} (A2);
	        \draw[line width=\lineWidth,>-<] (A2) to node[midway, below] {\emph{$-e_3$}} (A1);
	        \draw[line width=\lineWidth,>->] (A2) to node[midway, below] {\emph{$+e_5$}} (A3);
	        \draw[line width=\lineWidth,>->] (A1) to node[right] {$+e_{10}$} (A3);
	      \end{pgfonlayer}
	    \end{tikzpicture}
		\vspace{1cm}
		\caption{A bicyclic graph $\Gamma_2$ that contains two odd cycles.}\label{gamma3}
	\end{figure}
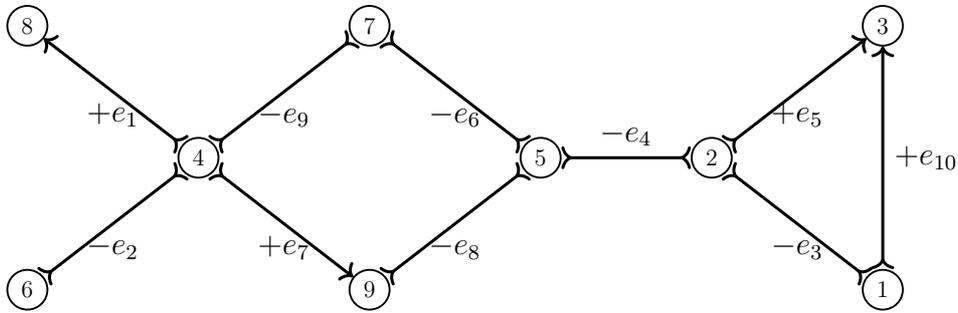

	Spanning subgraphs on $n=9$ edges of $\Gamma_3$ are depicted in Figure \ref{fig:44}.
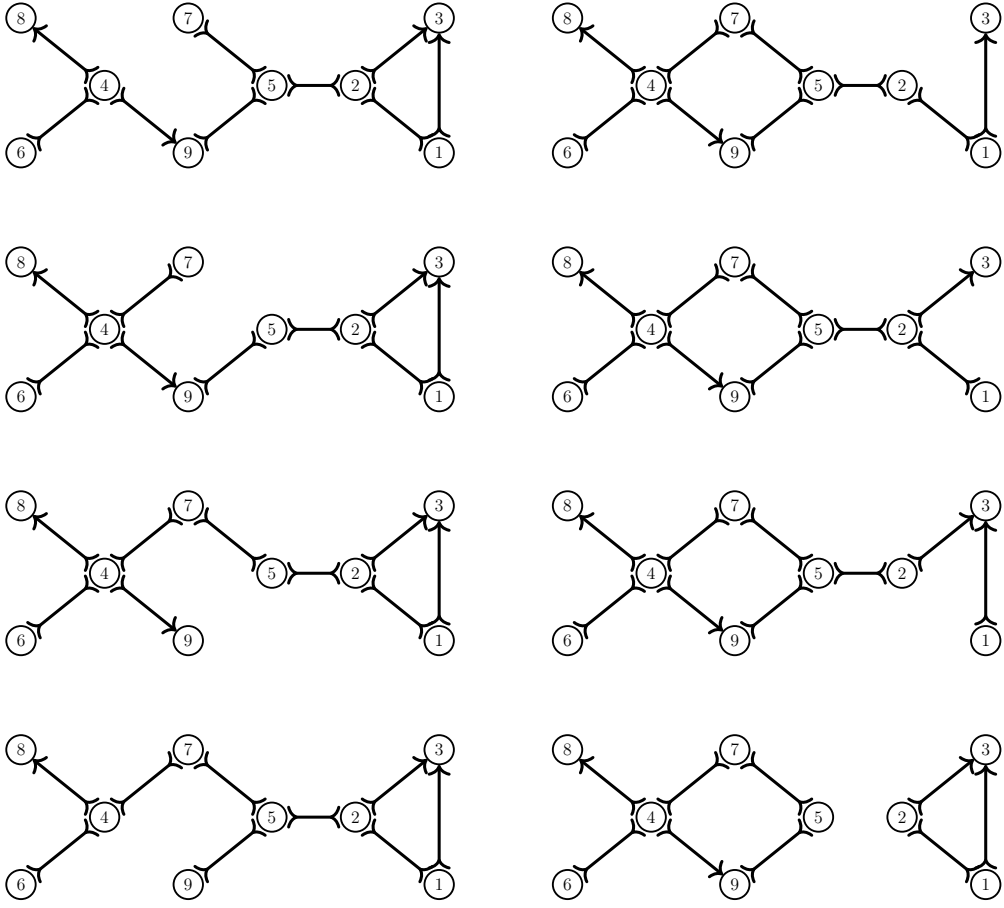
\begin{figure}[h]
	\centering
		\def\dis{3cm}
	    \def\NodeSC{0.55}
	    \def\Hdis{4.0}
	    \def\Vdis{3.25}
	    \begin{tikzpicture}[line width=0.25mm, baseline, scale=0.55, descr/.style={sloped,below}]
	      \node[whiteDot,scale=\NodeSC] (A6) at (0,0) {\bf $6$};
		  	\path (A6) +(90:\Vdis) node[whiteDot,scale=\NodeSC] (A8) {\bf $8$};
		  	\path (A6) +(0:\Hdis) node[whiteDot,scale=\NodeSC] (A9) {\bf $9$};
		  	\path (A8) +(0:\Hdis) node[whiteDot,scale=\NodeSC] (A7) {\bf $7$};
		  	\path (A7) +(0:1.5*\Hdis) node[whiteDot,scale=\NodeSC] (A3) {\bf $3$};
		  	\path (A9) +(0:1.5*\Hdis) node[whiteDot,scale=\NodeSC] (A1) {\bf $1$};
			\path (A6) +(0:0.5*\Hdis) node[] (A44) {};
		  	\path (A44) +(90:0.5*\Vdis) node[whiteDot,scale=\NodeSC] (A4) {\bf $4$};
		  	\path (A4) +(0:\Hdis) node[whiteDot,scale=\NodeSC] (A5) {\bf $5$};
		  	\path (A5) +(0:0.5*\Hdis) node[whiteDot,scale=\NodeSC] (A2) {\bf $2$};
	      \begin{pgfonlayer}{background}
	        \draw[line width=\lineWidth,>-<] (A4) to node {} (A6);
	        %\draw[line width=\lineWidth,>-<] (A4) to node {} (A7);
	        \draw[line width=\lineWidth,>->] (A4) to node {} (A8);
	        \draw[line width=\lineWidth,>->] (A4) to node {} (A9);
	        \draw[line width=\lineWidth,>-<] (A5) to node {} (A7);
	        \draw[line width=\lineWidth,>-<] (A5) to node {} (A9);
	        \draw[line width=\lineWidth,>-<] (A5) to node {} (A2);
	        \draw[line width=\lineWidth,>-<] (A2) to node {} (A1);
	        \draw[line width=\lineWidth,>->] (A2) to node {} (A3);
	        \draw[line width=\lineWidth,>->] (A1) to node {} (A3);
	      \end{pgfonlayer}
	    \end{tikzpicture}
		\hspace{1cm}
	    \begin{tikzpicture}[line width=0.25mm, baseline, scale=0.55, descr/.style={sloped,below}]
	      \node[whiteDot,scale=\NodeSC] (A6) at (0,0) {\bf $6$};
		  	\path (A6) +(90:\Vdis) node[whiteDot,scale=\NodeSC] (A8) {\bf $8$};
		  	\path (A6) +(0:\Hdis) node[whiteDot,scale=\NodeSC] (A9) {\bf $9$};
		  	\path (A8) +(0:\Hdis) node[whiteDot,scale=\NodeSC] (A7) {\bf $7$};
		  	\path (A7) +(0:1.5*\Hdis) node[whiteDot,scale=\NodeSC] (A3) {\bf $3$};
		  	\path (A9) +(0:1.5*\Hdis) node[whiteDot,scale=\NodeSC] (A1) {\bf $1$};
			\path (A6) +(0:0.5*\Hdis) node[] (A44) {};
		  	\path (A44) +(90:0.5*\Vdis) node[whiteDot,scale=\NodeSC] (A4) {\bf $4$};
		  	\path (A4) +(0:\Hdis) node[whiteDot,scale=\NodeSC] (A5) {\bf $5$};
		  	\path (A5) +(0:0.5*\Hdis) node[whiteDot,scale=\NodeSC] (A2) {\bf $2$};
	      \begin{pgfonlayer}{background}
	        \draw[line width=\lineWidth,>-<] (A4) to node {} (A6);
	        \draw[line width=\lineWidth,>-<] (A4) to node {} (A7);
	        \draw[line width=\lineWidth,>->] (A4) to node {} (A8);
	        \draw[line width=\lineWidth,>->] (A4) to node {} (A9);
	        \draw[line width=\lineWidth,>-<] (A5) to node {} (A7);
	        \draw[line width=\lineWidth,>-<] (A5) to node {} (A9);
	        \draw[line width=\lineWidth,>-<] (A5) to node {} (A2);
	        \draw[line width=\lineWidth,>-<] (A2) to node {} (A1);
	        %\draw[line width=\lineWidth,>->] (A2) to node {} (A3);
	        \draw[line width=\lineWidth,>->] (A1) to node {} (A3);
	      \end{pgfonlayer}
	    \end{tikzpicture}

		\vspace{1cm}
		
	    \begin{tikzpicture}[line width=0.25mm, baseline, scale=0.55, descr/.style={sloped,below}]
	      \node[whiteDot,scale=\NodeSC] (A6) at (0,0) {\bf $6$};
		  	\path (A6) +(90:\Vdis) node[whiteDot,scale=\NodeSC] (A8) {\bf $8$};
		  	\path (A6) +(0:\Hdis) node[whiteDot,scale=\NodeSC] (A9) {\bf $9$};
		  	\path (A8) +(0:\Hdis) node[whiteDot,scale=\NodeSC] (A7) {\bf $7$};
		  	\path (A7) +(0:1.5*\Hdis) node[whiteDot,scale=\NodeSC] (A3) {\bf $3$};
		  	\path (A9) +(0:1.5*\Hdis) node[whiteDot,scale=\NodeSC] (A1) {\bf $1$};
			\path (A6) +(0:0.5*\Hdis) node[] (A44) {};
		  	\path (A44) +(90:0.5*\Vdis) node[whiteDot,scale=\NodeSC] (A4) {\bf $4$};
		  	\path (A4) +(0:\Hdis) node[whiteDot,scale=\NodeSC] (A5) {\bf $5$};
		  	\path (A5) +(0:0.5*\Hdis) node[whiteDot,scale=\NodeSC] (A2) {\bf $2$};
	      \begin{pgfonlayer}{background}
	        \draw[line width=\lineWidth,>-<] (A4) to node {} (A6);
	        \draw[line width=\lineWidth,>-<] (A4) to node {} (A7);
	        \draw[line width=\lineWidth,>->] (A4) to node {} (A8);
	        \draw[line width=\lineWidth,>->] (A4) to node {} (A9);
	        %\draw[line width=\lineWidth,>-<] (A5) to node {} (A7);
	        \draw[line width=\lineWidth,>-<] (A5) to node {} (A9);
	        \draw[line width=\lineWidth,>-<] (A5) to node {} (A2);
	        \draw[line width=\lineWidth,>-<] (A2) to node {} (A1);
	        \draw[line width=\lineWidth,>->] (A2) to node {} (A3);
	        \draw[line width=\lineWidth,>->] (A1) to node {} (A3);
	      \end{pgfonlayer}
	    \end{tikzpicture}
		\hspace{1cm}
	    \begin{tikzpicture}[line width=0.25mm, baseline, scale=0.55, descr/.style={sloped,below}]
	      \node[whiteDot,scale=\NodeSC] (A6) at (0,0) {\bf $6$};
		  	\path (A6) +(90:\Vdis) node[whiteDot,scale=\NodeSC] (A8) {\bf $8$};
		  	\path (A6) +(0:\Hdis) node[whiteDot,scale=\NodeSC] (A9) {\bf $9$};
		  	\path (A8) +(0:\Hdis) node[whiteDot,scale=\NodeSC] (A7) {\bf $7$};
		  	\path (A7) +(0:1.5*\Hdis) node[whiteDot,scale=\NodeSC] (A3) {\bf $3$};
		  	\path (A9) +(0:1.5*\Hdis) node[whiteDot,scale=\NodeSC] (A1) {\bf $1$};
			\path (A6) +(0:0.5*\Hdis) node[] (A44) {};
		  	\path (A44) +(90:0.5*\Vdis) node[whiteDot,scale=\NodeSC] (A4) {\bf $4$};
		  	\path (A4) +(0:\Hdis) node[whiteDot,scale=\NodeSC] (A5) {\bf $5$};
		  	\path (A5) +(0:0.5*\Hdis) node[whiteDot,scale=\NodeSC] (A2) {\bf $2$};
	      \begin{pgfonlayer}{background}
	        \draw[line width=\lineWidth,>-<] (A4) to node {} (A6);
	        \draw[line width=\lineWidth,>-<] (A4) to node {} (A7);
	        \draw[line width=\lineWidth,>->] (A4) to node {} (A8);
	        \draw[line width=\lineWidth,>->] (A4) to node {} (A9);
	        \draw[line width=\lineWidth,>-<] (A5) to node {} (A7);
	        \draw[line width=\lineWidth,>-<] (A5) to node {} (A9);
	        \draw[line width=\lineWidth,>-<] (A5) to node {} (A2);
	        \draw[line width=\lineWidth,>-<] (A2) to node {} (A1);
	        \draw[line width=\lineWidth,>->] (A2) to node {} (A3);
	        %\draw[line width=\lineWidth,>->] (A1) to node {} (A3);
	      \end{pgfonlayer}
	    \end{tikzpicture}

		\vspace{1cm}
		
	    \begin{tikzpicture}[line width=0.25mm, baseline, scale=0.55, descr/.style={sloped,below}]
	      \node[whiteDot,scale=\NodeSC] (A6) at (0,0) {\bf $6$};
		  	\path (A6) +(90:\Vdis) node[whiteDot,scale=\NodeSC] (A8) {\bf $8$};
		  	\path (A6) +(0:\Hdis) node[whiteDot,scale=\NodeSC] (A9) {\bf $9$};
		  	\path (A8) +(0:\Hdis) node[whiteDot,scale=\NodeSC] (A7) {\bf $7$};
		  	\path (A7) +(0:1.5*\Hdis) node[whiteDot,scale=\NodeSC] (A3) {\bf $3$};
		  	\path (A9) +(0:1.5*\Hdis) node[whiteDot,scale=\NodeSC] (A1) {\bf $1$};
			\path (A6) +(0:0.5*\Hdis) node[] (A44) {};
		  	\path (A44) +(90:0.5*\Vdis) node[whiteDot,scale=\NodeSC] (A4) {\bf $4$};
		  	\path (A4) +(0:\Hdis) node[whiteDot,scale=\NodeSC] (A5) {\bf $5$};
		  	\path (A5) +(0:0.5*\Hdis) node[whiteDot,scale=\NodeSC] (A2) {\bf $2$};
	      \begin{pgfonlayer}{background}
	        \draw[line width=\lineWidth,>-<] (A4) to node {} (A6);
	        \draw[line width=\lineWidth,>-<] (A4) to node {} (A7);
	        \draw[line width=\lineWidth,>->] (A4) to node {} (A8);
	        \draw[line width=\lineWidth,>->] (A4) to node {} (A9);
	        \draw[line width=\lineWidth,>-<] (A5) to node {} (A7);
	        %\draw[line width=\lineWidth,>-<] (A5) to node {} (A9);
	        \draw[line width=\lineWidth,>-<] (A5) to node {} (A2);
	        \draw[line width=\lineWidth,>-<] (A2) to node {} (A1);
	        \draw[line width=\lineWidth,>->] (A2) to node {} (A3);
	        \draw[line width=\lineWidth,>->] (A1) to node {} (A3);
	      \end{pgfonlayer}
	    \end{tikzpicture}
		\hspace{1cm}
	    \begin{tikzpicture}[line width=0.25mm, baseline, scale=0.55, descr/.style={sloped,below}]
	      \node[whiteDot,scale=\NodeSC] (A6) at (0,0) {\bf $6$};
		  	\path (A6) +(90:\Vdis) node[whiteDot,scale=\NodeSC] (A8) {\bf $8$};
		  	\path (A6) +(0:\Hdis) node[whiteDot,scale=\NodeSC] (A9) {\bf $9$};
		  	\path (A8) +(0:\Hdis) node[whiteDot,scale=\NodeSC] (A7) {\bf $7$};
		  	\path (A7) +(0:1.5*\Hdis) node[whiteDot,scale=\NodeSC] (A3) {\bf $3$};
		  	\path (A9) +(0:1.5*\Hdis) node[whiteDot,scale=\NodeSC] (A1) {\bf $1$};
			\path (A6) +(0:0.5*\Hdis) node[] (A44) {};
		  	\path (A44) +(90:0.5*\Vdis) node[whiteDot,scale=\NodeSC] (A4) {\bf $4$};
		  	\path (A4) +(0:\Hdis) node[whiteDot,scale=\NodeSC] (A5) {\bf $5$};
		  	\path (A5) +(0:0.5*\Hdis) node[whiteDot,scale=\NodeSC] (A2) {\bf $2$};
	      \begin{pgfonlayer}{background}
	        \draw[line width=\lineWidth,>-<] (A4) to node {} (A6);
	        \draw[line width=\lineWidth,>-<] (A4) to node {} (A7);
	        \draw[line width=\lineWidth,>->] (A4) to node {} (A8);
	        \draw[line width=\lineWidth,>->] (A4) to node {} (A9);
	        \draw[line width=\lineWidth,>-<] (A5) to node {} (A7);
	        \draw[line width=\lineWidth,>-<] (A5) to node {} (A9);
	        \draw[line width=\lineWidth,>-<] (A5) to node {} (A2);
	        %\draw[line width=\lineWidth,>-<] (A2) to node {} (A1);
	        \draw[line width=\lineWidth,>->] (A2) to node {} (A3);
	        \draw[line width=\lineWidth,>->] (A1) to node {} (A3);
	      \end{pgfonlayer}
	    \end{tikzpicture}

		\vspace{1cm}
		
	    \begin{tikzpicture}[line width=0.25mm, baseline, scale=0.55, descr/.style={sloped,below}]
	      \node[whiteDot,scale=\NodeSC] (A6) at (0,0) {\bf $6$};
		  	\path (A6) +(90:\Vdis) node[whiteDot,scale=\NodeSC] (A8) {\bf $8$};
		  	\path (A6) +(0:\Hdis) node[whiteDot,scale=\NodeSC] (A9) {\bf $9$};
		  	\path (A8) +(0:\Hdis) node[whiteDot,scale=\NodeSC] (A7) {\bf $7$};
		  	\path (A7) +(0:1.5*\Hdis) node[whiteDot,scale=\NodeSC] (A3) {\bf $3$};
		  	\path (A9) +(0:1.5*\Hdis) node[whiteDot,scale=\NodeSC] (A1) {\bf $1$};
			\path (A6) +(0:0.5*\Hdis) node[] (A44) {};
		  	\path (A44) +(90:0.5*\Vdis) node[whiteDot,scale=\NodeSC] (A4) {\bf $4$};
		  	\path (A4) +(0:\Hdis) node[whiteDot,scale=\NodeSC] (A5) {\bf $5$};
		  	\path (A5) +(0:0.5*\Hdis) node[whiteDot,scale=\NodeSC] (A2) {\bf $2$};
	      \begin{pgfonlayer}{background}
	        \draw[line width=\lineWidth,>-<] (A4) to node {} (A6);
	        \draw[line width=\lineWidth,>-<] (A4) to node {} (A7);
	        \draw[line width=\lineWidth,>->] (A4) to node {} (A8);
	        %\draw[line width=\lineWidth,>->] (A4) to node {} (A9);
	        \draw[line width=\lineWidth,>-<] (A5) to node {} (A7);
	        \draw[line width=\lineWidth,>-<] (A5) to node {} (A9);
	        \draw[line width=\lineWidth,>-<] (A5) to node {} (A2);
	        \draw[line width=\lineWidth,>-<] (A2) to node {} (A1);
	        \draw[line width=\lineWidth,>->] (A2) to node {} (A3);
	        \draw[line width=\lineWidth,>->] (A1) to node {} (A3);
	      \end{pgfonlayer}
	    \end{tikzpicture}
		\hspace{1cm}
	    \begin{tikzpicture}[line width=0.25mm, baseline, scale=0.55, descr/.style={sloped,below}]
	      \node[whiteDot,scale=\NodeSC] (A6) at (0,0) {\bf $6$};
		  	\path (A6) +(90:\Vdis) node[whiteDot,scale=\NodeSC] (A8) {\bf $8$};
		  	\path (A6) +(0:\Hdis) node[whiteDot,scale=\NodeSC] (A9) {\bf $9$};
		  	\path (A8) +(0:\Hdis) node[whiteDot,scale=\NodeSC] (A7) {\bf $7$};
		  	\path (A7) +(0:1.5*\Hdis) node[whiteDot,scale=\NodeSC] (A3) {\bf $3$};
		  	\path (A9) +(0:1.5*\Hdis) node[whiteDot,scale=\NodeSC] (A1) {\bf $1$};
			\path (A6) +(0:0.5*\Hdis) node[] (A44) {};
		  	\path (A44) +(90:0.5*\Vdis) node[whiteDot,scale=\NodeSC] (A4) {\bf $4$};
		  	\path (A4) +(0:\Hdis) node[whiteDot,scale=\NodeSC] (A5) {\bf $5$};
		  	\path (A5) +(0:0.5*\Hdis) node[whiteDot,scale=\NodeSC] (A2) {\bf $2$};
	      \begin{pgfonlayer}{background}
	        \draw[line width=\lineWidth,>-<] (A4) to node {} (A6);
	        \draw[line width=\lineWidth,>-<] (A4) to node {} (A7);
	        \draw[line width=\lineWidth,>->] (A4) to node {} (A8);
	        \draw[line width=\lineWidth,>->] (A4) to node {} (A9);
	        \draw[line width=\lineWidth,>-<] (A5) to node {} (A7);
	        \draw[line width=\lineWidth,>-<] (A5) to node {} (A9);
	        %\draw[line width=\lineWidth,>-<] (A5) to node {} (A2);
	        \draw[line width=\lineWidth,>-<] (A2) to node {} (A1);
	        \draw[line width=\lineWidth,>->] (A2) to node {} (A3);
	        \draw[line width=\lineWidth,>->] (A1) to node {} (A3);
	      \end{pgfonlayer}
	    \end{tikzpicture}
		\caption{Spanning $TU$ subgraphs of $H$}
	\end{figure}\label{fig:44}

Therefore,
$$\vol^2 N=7\cdot 4+4^2=44$$ and
$$(N^\dagger)_{53}=\frac{1}{44}\big(-4\cdot\frac{1}{2}-4\cdot\frac{1}{2}-4\cdot 1-4\cdot\frac{1}{2}-4\cdot 1-4\cdot\frac{1}{2}-4^2\cdot\frac{1}{2}\big)=-\frac{24}{44}=-\frac{6}{11}.$$

\end{example}

\begin{remark}
If $\Gamma$ is a connected signed graph with all positive edges, then general formula is matching the known formula from \cite{bap}. If $\Gamma$ is a connected signed graph with all negative edges, then general formula matches the one given in \cite{AOM}.
\end{remark}

\section{Moore-Penrose inverse of the Laplacian matrix of a signed graph}\label{sec:5}

\noindent In this final section we consider the Moore-Penrose inverse of  the Laplacian matrix of a connected signed graph. 
\begin{theorem}
Let $\Gamma=(G,\sigma)$ be a connected signed graph on $n\geq 2$ vertices $1,2,\ldots,n$ and $m$ edges $e_1,e_2,\ldots,e_m$ with the Laplacian matrix $L$. 
\begin{enumerate}
\item[(a)] If $\Gamma$ is balanced, then 
\begin{equation}
(L^\dagger)_{ij}= \frac{1}{n^2\tau(\Gamma)^2}\sum_{k=1}^m\sum_{\substack{T_r, T_s\in \mathcal{S}(\Gamma)\\e_k\in T_r, T_s}} \sgn(P_{T_r}(e_k-i))\sgn(P_{T_s}(e_k-j))\Psi(e_k,i,j),
\end{equation}
where
\begin{equation}
\Psi(e_k,i,j)=
 \begin{cases}
 |T^{T_r}_h(e_k)| |T^{T_s}_h(e_k)|& \text{ if } i \in T^{T_r}_t(e_k)\text { and } j\in T^{T_s}_t(e_k) \\
 |T^{T_r}_t(e_k)| |T^{T_s}_t(e_k)|& \text{ if } i \in T^{T_r}_h(e_k)\text { and } j\in T^{T_s}_h(e_k)\\
 -\sgn(e_k)|T^{T_r}_h(e_k)| |T^{T_s}_t(e_k)|& \text{ otherwise}.
 \end{cases} 
\end{equation}

\item[(b)] If $\Gamma$ is unbalanced, then
\begin{equation}
(L^{-1})_{ij}= \frac{1}{(\sum_{H\in \mathcal{TU}_n(\Gamma)} 4^{c(H)})^2}\sum_{k=1}^m\sum_{\substack{H\in \mathcal{TU}_n(\Gamma)\\e_k\in H}} 4^{c(H)} (N_H^+)_{
ki}\sum_{\substack{H\in \mathcal{TU}_n(\Gamma)\\e_k\in H}} 4^{c(H)} (N_H^+)_{kj},
\end{equation}
where $\mathcal{TU}_n(\Gamma)$ is the set of all spanning $TU$-subgraphs $H$ of $\Gamma$ with $n$ edges consisting of $c(H)$ unbalanced unicyclic graphs $U^H_1,U^H_2,\ldots,U^H_{c(H)}$ with cycles $C^H_1,C^H_2,\ldots,C^H_{c(H)}$ respectively  and when $e_i\in U^H_r$ for some $r$, then
\begin{equation}\label{Unbalanced G N inverse2}
(N_H^\dagger)_{ij}= \begin{cases}
\frac{1}{2}\sgn\left(P_{e_i^t-j}\right) & \text{ if } e_i \in C^H_r\\
0 & \text{ if } e_i \not \in C_r \text{ and } j \in U^H_r\setminus e_i [C^H_r]\\
-\sgn(P_{e_i-j}) & \text{ if } e_i \in E^+(U^H_r)\setminus E(C^H_r) \text{ and } j \not \in U^H_r\setminus e_i [C^H_r] \text{ and } j\in (U^H_r)_h(e_i)\\
\sgn(P_{e_i-j}) & \text{ otherwise, }
\end{cases} 
\end{equation}
where $P_{e_i^t-j}$ denotes the path  between the tail of $e_i\in C^H_r$ and $j$ in the tree $U_r^H\setminus e_i$.

\end{enumerate}
\end{theorem}

\begin{proof}
The expressions follow from
\[L=NN^\intercal \implies L^\dagger=(N^\dagger)^\intercal N^\dagger.\]
(a) For $N^\dagger=[n_{ij}]$,
$$(L^\dagger)_{ij}=\sum_{k=1}^m n_{ki}n_{kj}=\frac{1}{n\tau(\Gamma)}\sum_{T\in S(\Gamma)\atop e_k\in T} ({N_T^\dagger})_{ki}\frac{1}{n\tau(\Gamma)}\sum_{T\in S(\Gamma)\atop e_k\in T} ({N_T^\dagger})_{kj}=\frac{1}{n^2\tau(\Gamma)^2}\sum_{T_t,T_s\in S(\Gamma)\atop e_k\in T_r, T_s} ({N_T^\dagger})_{ki}({N_T^\dagger})_{kj}$$
Taking into account the results on $N_T^\dagger$ given in Theorem \ref{tree} we arrive to 
$$({N_T^\dagger})_{ki}({N_T^\dagger})_{kj}=\sgn(P_{T_r}(e_k-i))\sgn(P_{T_s}(e_k-j))\Psi(e_k,i,j),$$ as desired.

The proof of (b) follows in a similar way.
\end{proof}

\section{Open problems}\label{sec:6}
\noindent It can be verified that for a signed graph $\Gamma$,
$$\ell^\dagger_{ii}+\ell^\dagger_{jj}-2\sgn(P_{i-j})\ell^\dagger_{ij}=d(i,j).$$ This property can lead to the definition of resistance matrix  $R$, $r_{ij}= \ell^\dagger_{ii}+\ell^\dagger_{jj}-2\sgn(P_{i-j})\ell^\dagger_{ij}$ of signed graphs. It might  be worth considering  the properties of this matrix as well as its rise to the Kemany's constant of signed graphs (as suggested in \cite{sh}).

Barik et al. introduced some complex matrices associated with a multidigraph via a certain weighted graph \cite{Barik1}. It will be interesting to find out what their Moore-Penrose inverses tell us about the associated multidigraph.
\begin{question}
Find combinatorial formulas for the following matrices for a connected multidigraph:
(a) complex incidence and  complex Laplacian matrices, and 
(b) complex signless incidence and complex signless Laplacian matrices.
\end{question}

\bigskip

\noindent{\bf\large Acknowledgments}\\
The authors would like to thank Thomas Zaslavsky for inspiring them to look into signed graphs in the study of Moore-Penrose inverses of the associated matrices.

The research of the second author is partially supported by the Science Fund of the Republic of Serbia; grant number 7749676: Spectrally Constrained Signed Graphs with Applications in Coding Theory and Control Theory -- SCSG-ctct.

%\newpage

\end{document}